\documentclass[12pt]{amsart}
\usepackage{fullpage}
\usepackage[utf8]{inputenc}
\usepackage[english]{babel}
\usepackage{mathtools}
\usepackage{amssymb}
\usepackage{amsthm}
\usepackage{mathrsfs}
\usepackage{nicefrac}
\usepackage{upgreek}
\usepackage[T1]{fontenc}
\usepackage{graphicx}
\usepackage{booktabs}
\usepackage{graphics}
\usepackage{hyperref}

\usepackage{tikz}
        \usetikzlibrary{positioning}
				\usetikzlibrary{arrows.meta}
\usepackage{xargs}

\usepackage{tikz-cd}
\tikzcdset{%
    triple line/.code={\tikzset{%
        double equal sign distance, % replace by double distance = 'measure' 
        double=\pgfkeysvalueof{/tikz/commutative diagrams/background color}}},
    quadruple line/.code={\tikzset{%
        double equal sign distance, % replace by double distance = 'measure'
        double=\pgfkeysvalueof{/tikz/commutative diagrams/background color}}},
    Rrightarrow/.code={\tikzcdset{triple line}\pgfsetarrows{tikzcd implies cap-tikzcd implies}},
    RRightarrow/.code={\tikzcdset{quadruple line}\pgfsetarrows{tikzcd implies cap-tikzcd implies}}
}

\newlength{\myline}% line thickness
\newcommandx*{\triplearrow}[4][1=0, 2=1]{% #1 = shorten left (optional), #2 = shorten right (optionsl),
  \draw[line width=\myline,double distance=3\myline,#3] #4;
  \draw[line width=\myline,shorten <=#1\myline,shorten >=#2\myline,#3] #4;
}
\newcommandx*{\quadarrow}[4][1=0, 2=2.5]{% #1 = shorten left (optional), #2 = shorten right (optionsl),
  \draw[line width=\myline,double distance=5\myline,#3] #4;
  \draw[line width=\myline,double distance=\myline,shorten <=#1\myline,shorten >=#2\myline,#3] #4;
}

\usetikzlibrary{arrows.meta}
\tikzset{%
  >={Latex[width=2mm,length=2mm]},
            base/.style = {rectangle, rounded corners, draw=black,
                           minimum width=0.5cm, minimum height=0.25cm,
                           text centered, font=\tiny},
  activityStarts/.style = {base, fill=blue!30},
       startstop/.style = {base, fill=red!30},
    activityRuns/.style = {base, fill=green!30},
         process/.style = {base, minimum width=1cm, fill=orange!15,
                           font=\tiny},}
\usepackage[all]{xy}
\usepackage{graphics}
\usepackage{color}
\usepackage[T1]{fontenc}
\usepackage{enumitem}

\makeatletter
\newtheorem*{rep@theorem}{\rep@title}
\newcommand{\newreptheorem}[2]{%
\newenvironment{rep#1}[1]{%
 \def\rep@title{#2 \ref{##1}}%
 \begin{rep@theorem}}%
 {\end{rep@theorem}}}
\makeatother

\newcommand{\dl}{\langle\langle}
\newcommand{\dr}{\rangle\rangle}
\newcommand{\card}[1]{|#1|}

\newcommand{\C}[1]{{\mathcal #1}}

\newtheorem{Theorem}{Theorem}[section]
\newreptheorem{Theorem}{Theorem}
\newtheorem{mTheorem}{Theorem}
\newtheorem{Lemma}[Theorem]{Lemma}
\newreptheorem{Lemma}{Lemma}

\newreptheorem{Corollary}{Corollary}
\newtheorem{Proposition}[Theorem]{Proposition}
\newreptheorem{Proposition}{Proposition}

\newtheorem{mConjecture}{Conjecture}
\theoremstyle{definition}\newtheorem{mydef}[Theorem]{Definition}
\theoremstyle{remark}\newtheorem{remark}[Theorem]{Remark}
\theoremstyle{definition}
\makeatletter\@addtoreset{case}{example}\makeatother
\theoremstyle{definition}

\begin{document}

\title{Bounded generation for congruence subgroups of ${\rm Sp}_4(R)$}

\author{Alexander A. Trost}
\address{University of Aberdeen}
\email{r01aat17@abdn.ac.uk}

\begin{abstract}
This paper describes a bounded generation result concerning the minimal natural number $K$ such that for 
$Q(C_2,2R):=\{A\varepsilon_{\phi}(2x)A^{-1}|x\in R,A\in{\rm Sp}_4(R),\phi\in C_2\}$, one has $N_{C_2,2R}=\{X_1\cdots X_K|\forall 1\leq i\leq K:X_i\in Q(C_2,2R)\}$
for certain rings of algebraic integers $R$ and the principal congruence subgroup $N_{C_2,2R}$ in ${\rm Sp}_4(R).$ This gives an explicit version of an abstract bounded generation result of a similar type as presented by Morris \cite[Theorem~6.1(1)]{MR2357719}. Furthermore, the result presented does not depend on several number-theoretic quantities unlike Morris' result. Using this bounded generation result, we further give explicit bounds for the strong boundedness of ${\rm Sp}_4(R)$ for certain examples of rings $R,$ thereby giving explicit versions of earlier strong boundedness results in \cite{General_strong_bound}. We further give a classification of normally generating subsets of ${\rm Sp}_4(R)$ for $R$ a ring of algebraic integers. 
\end{abstract}

\maketitle

\section*{Introduction}

Bounded generation is a classic topic in the discussion of non-uniform lattices. For a group $G,$ it is usually about finding a collection $Z_1,\dots,Z_N$ of 'nice' (often cyclic) subgroups of $G$ such that $G=Z_1\cdots Z_N$. There have been some fascinating papers about this topic in the last decade, for example Morris' paper \cite{MR2357719} about bounded generation for ${\rm SL}_n(R)$ from a model theoretic viewpoint, Rapinchuk, Morgan and Sury's paper \cite{MR3892969} about bounded generation of ${\rm SL}_2(R)$ for $R$ a ring of S-algebraic integers with infinitely many units and Nica's paper \cite{Nica} about bounded generation of ${\rm SL}_n(\mathbb{F}[T])$ by root elements. 

A new variant of bounded generation that has emerged is the study of bounded generation not by way of subgroups but by collections of conjugacy classes: That is for a group $G$ and a subset $T$ of $G$ that generates $G$ and is closed under conjugation, one asks if there is a (minimal) natural number $L:=L(G,T)$ such that 
\begin{equation*}
G=\{t_1\cdots t_L|\forall 1\leq i\leq L:t_i\in T\cup T^{-1}\}
\end{equation*}
and if such a $L$ exists how it depends on $T$ and on $G.$ This type of problem has been studied for various different groups like diffeomorphism groups by Burago, Ivanov and Polterovich \cite{MR2509711}, finite simple groups by Liebeck, Lawther and Shalev \cite{MR1865975},\cite{lawther1998diameter} and Lie Groups and linear algebraic groups by Kedra, Libman and Martin \cite{KLM}. Further, for $\Phi$ an irreducible root system of rank at least $2$ and $R$ a ring of algebraic integers, one can use classical bounded generation results for $G(\Phi,R)$ to see that $L(G(\Phi,R),T)$ always exists as observed by Burago, Ivanov and Polterovich \cite[Example~1.6]{MR2509711} and Kedra and Gal \cite[Theorem~1.1]{MR2819193}. But there are choices for $T$ and $G$ that give rise to new questions about bounded generation and in this paper we are interested in two in particular: First, the case of $G$ being the principal congruence subgroup $N_{I,\Phi}$ in $G(\Phi,R)$ for $I$ a proper ideal in $R$ and $T=Q(\Phi,I):=\{A\varepsilon_{\phi}(x)A^{-1}|x\in I,A\in G(\Phi,R),\phi\in \Phi\}$. Note that this uses that $Q(\Phi,I)$ generates $N_{I,\Phi},$ which is true according to Milnor's, Serre's and Bass' solution for the Congruence subgroup problem \cite[Theorem~3.6, Corollary~12.5]{MR244257}, if $R$ has infinitely many units or $I=2R,$ which are the two cases we are interested in in this paper. A result by Morris \cite[Theorem~6.1(1)]{MR2357719} implies that the quantity $L(N_{I,\Phi},Q(\Phi,I))=:K(I,\Phi)$ exists and has an upper bound depending on the number of generators of the ideal $I$, ${\rm rank}(\Phi),|R/I|$ as well as $[K:\mathbb{Q}]$ for $K$ the number field containing $R.$ Crucially, this upper bound does not depend on the algebraic structure of $R/I$ or $R$, but only on some numbers associated with the tuple $(\Phi,R,I,K)$. However, we believe that the situation is better than described by Morris, if $R$ has infinitely many units: 

\begin{mConjecture}\label{fundamental_conjecture}
Let $R$ be a ring of algebraic integers with infinitely many units, $I$ a non-trivial ideal in $R$, $\Phi$ an irreducible root system of rank at least $2$. Then there is a minimal constant $K:=K(\Phi)$ proportional to ${\rm rank}(\Phi)$ and independent of $R$ and $I$ such that $N_{I,\Phi}=Q(\Phi,I)^K$.
\end{mConjecture}

We want to provide supporting evidence for this conjecture in the case that $I=2R$ is a prime ideal, namely in the sub-case that each element of $R$ is a sum of a unit in $R$ and an element of $2R.$ I call such rings \textit{$2R$-pseudo-good.} The motivating example is, of course, the ring of integers $\mathbb{Z},$ but we will show that some rings of quadratic and cubic integers $R$ with $2R$ prime have this property as well. We will give bounds on $K(C_2,2R)$ for such rings:

\begin{mTheorem}\label{fundamental_conjecture_partial_result}
Let $R$ be a $2R$-pseudo-good ring of S-algebraic integers. 
\begin{enumerate}
\item{If $R$ has infinitely many units, than $K(C_2,2R)\leq 46$.}
\item{If $R$ is a principal ideal domain, then $K(C_2,2R)\leq 646.$}
\end{enumerate}
\end{mTheorem}

This is essentially shown by determining how 'expensive' a classical proof for the existence of the Bruhat decomposition on ${\rm Sp}_4(R/2R)$ is in terms of $Q(C_2,2R)$ when forced on ${\rm Sp}_4(R)$. Theorem~\ref{fundamental_conjecture_partial_result} does not depend on the ring $R$ and the corresponding field extension $K|\mathbb{Q}$ at least if $R$ has infinitely many units, which supports Conjecture~\ref{fundamental_conjecture}.   

Second, we study the case that $T$ is a union of finitely many conjugacy classes $C_1,\dots,C_k$ generating $G(\Phi,R)$. This problem has been studied in recent years, too. For example, we have shown in \cite[Theorem~5.13]{General_strong_bound} that independently of the specific conjugacy classes $C_1,\dots,C_k$ the quantity $L(G(\Phi,R),T)$ has an upper bound $L_k$ only depending on $k,\Phi$ and $R.$ For a group $G$ and $k\in\mathbb{N}$, the minimal $L_k$ that works for all conjugacy classes $C_1,\dots,C_k$ generating $G$ is denoted by $\Delta_k(G),$ if it exists. In fact, our result \cite[Theorem~5.13]{General_strong_bound} shows that there is a $C(\Phi,R)\in\mathbb{N}$ with $k\leq\Delta_k(G(\Phi,R))\leq C(\Phi,R)k$ for $k$ sufficiently big. Further explicit bounds in \cite[Corollary~6.2]{KLM} and by myself \cite[Theorem~3]{explicit_strong_bound_sp_2n} strongly suggest for $R$ a ring of integers with infinitely many units and $\Phi$ irreducible of rank at least $3$ that $\Delta_k(G(\Phi,R))$ has an upper bound proportional to ${\rm rank}(\Phi)^2\cdot k$ and a lower bound proportional to ${\rm rank}(\Phi)\cdot k$ with proportionality factors independent of $R.$ I believe that $\Delta_k(G(\Phi,R))$ has an upper bound proportional to ${\rm rank}(\Phi)\cdot k$, too, and a research strategy suggested to us by the anonymous referee of \cite{General_strong_bound} would imply this. This strategy however depends on Conjecture~\ref{fundamental_conjecture}.

Our results in \cite{General_strong_bound} showed further that the discussion of $\Delta_k(G)$ for $G={\rm Sp}_4(R)$ or $G_2(R)$ is more involved than for higher rank root systems $\Phi$ and involves certain number theoretic problems related to what I call \textit{bad primes of $R$}. Our result \cite[Theorem~5.13]{General_strong_bound} does still show that $\Delta_k({\rm Sp}_4(R))$ has upper and lower bounds proportional to $k$ for $k$ sufficiently big, but as of now explicit bounds on $\Delta_k({\rm Sp}_4(R))$ have not appeared in the literature. This is partly due to the fact giving explicit bounds for $\Delta_k({\rm Sp}_4(R))$ requires one to consider the value of $K(C_2,2R)$, because contrary to higher rank groups $G(\Phi,R)$, it is hard to construct all root elements of ${\rm Sp}_4(R)$ from a collection of generating conjugacy classes $T$. So Theorem~\ref{fundamental_conjecture_partial_result} can be used to derive explicit bounds on $\Delta_k({\rm Sp}_4(R))$ for $2R$-pseudo-good rings $R$:

\begin{mTheorem}\label{sp4_strong_boundedness_explicit}
Let $k\in\mathbb{N}$ be given, $p$ a prime greater than $3$. Further, let $D$ be a positive, square-free number with $D\equiv 5\text{ mod }8$ such that there are $a,b$ positive odd numbers with $b^2D=a^2\pm 4$ and such that $R_D$ the ring of algebraic integers in the number field $\mathbb{Q}[\sqrt{D}],$ is a principal ideal domain. Then
\begin{enumerate}
\item{$\Delta_k({\rm Sp}_4(R_D)\leq 4+17644k$,}
\item{$\Delta_k({\rm Sp}_4(\mathbb{Z}[p^{-1}]))\leq 5+17644k$,}
\item{$\Delta_k({\rm Sp}_4(\mathbb{Z}))\leq 5+248064k$ and}
\item{$\Delta_k\left({\rm Sp}_4(\mathbb{Z}[\frac{1+\sqrt{-3}}{2}])\right)\leq 4+248064k$ hold.}
\end{enumerate}
\end{mTheorem}

We also provide an improved version of \cite[Theorem~6.3]{General_strong_bound} to demonstrate that even for $k\geq r(R),$ the number of bad primes $r(R)$ of $R$ can not be ignored when determining $\Delta_k({\rm Sp}_4(R)):$

\begin{mTheorem}\label{lower_bounds_better}
Let $R$ be a ring of S-algebraic integers in a number field. Further let
\begin{equation*}
r:=r(R):=|\{\C P|\ \C P\text{ divides 2R, is a prime ideal and }R/\C P=\mathbb{F}_2\}|
\end{equation*}
be given. Then $\Delta_k({\rm Sp}_4(R))\geq 4k+r(R)$ for all $k\in\mathbb{N}$ with $k\geq r(R)$.
\end{mTheorem}

Lastly, we prove the following theorem classifying normally generating subsets of ${\rm Sp}_4(R),$ whose proof we promised in an earlier paper \cite[Corollary~6.8]{General_strong_bound}: 

\begin{mTheorem}\label{classifying_normal_generating_subsets_better}
Let $R$ be a ring of S-algebraic integers and $A:{\rm Sp}_4(R)\to {\rm Sp}_4(R)/[{\rm Sp}_4(R),{\rm Sp}_4(R)]$ the abelianization homomorphism. 
Then $S\subset {\rm Sp}_4(R)$ normally generates ${\rm Sp}_4(R)$ precisely if $\Pi(S)=\emptyset$ and $A(S)$ generates ${\rm Sp}_4(R)/[{\rm Sp}_4(R),{\rm Sp}_4(R)].$ 
\end{mTheorem}

This paper is divided into six sections: In the first section, we introduce necessary definitions and notations. The second section explains how to provide values for $K(C_2,2R)$ for $2R$-pseudo-good rings. The third section gives some examples of $2R$-pseudo-good rings of algebraic integers and proves Theorem~\ref{fundamental_conjecture_partial_result}. In the fourth section, we restate and slightly rephrase \cite[Theorem~3.1]{General_strong_bound} in order to prove Theorem~\ref{sp4_strong_boundedness_explicit}. The fifth and sixth section prove Theorem~\ref{classifying_normal_generating_subsets_better} and Theorem~\ref{lower_bounds_better} respectively.  

\section*{Acknowledgments}

I want to thank Benjamin Martin for his continued support and advice. %This work was funded by Leverhulme Trust Research Project Grant RPG-2017-159. 

\section{Definitions}\label{definition}

Let $G$ be a group and $S$ a finite subset of $G$. In this paper $\|\cdot\|_S:G\to\mathbb{N}_0\cup\{+\infty\}$ denotes the word norm given by the conjugacy classes $C_G(S)$ on $G$, that is $\|1\|_S:=0$,
\begin{equation*}
\|X\|_S:=\min\{n\in\mathbb{N}|\exists A_1,\dots,A_n\in C_G(S\cup S^{-1}):X=A_1\cdots A_n\}
\end{equation*}
for $X\in\dl S\dr-\{1\}$ and $\|X\|_S:=+\infty$ for $X\notin\dl S\dr.$ We also set $B_S(k):=\{A\in G|\|A\|_S\leq k\}$ for $k\in\mathbb{N}$ and $\|G\|_S={\rm diam}(\|\cdot\|_S)$ of $G$ as the minimal $N\in\mathbb{N}$, such that $B_S(N)=G$ or as $+\infty$ if there is no such $N$. If $S=\{A\}$, then we write $\|\cdot\|_A$ instead of $\|\cdot\|_{\{A\}}$ and $B_A(k)$ instead of $B_{\{A\}}(k).$ 
Further define for $k\in\mathbb{N}$ the invariant 
\begin{equation*}
\Delta_k(G):=\sup\{{\rm diam}(\|\cdot\|_S)|\ S\subset G\text{ with }\card{S}\leq k\text{ and }\dl S\dr=G\}\in\mathbb{N}_0\cup\{+\infty\}
\end{equation*}
with $\Delta_k(G)$ defined as $-\infty$, if there is no normally generating set $S\subset G$ with $\card{S}\leq k.$ The group $G$ is called \textit{strongly bounded}, if $\Delta_k(G)$ is finite or $-\infty$ for all $k\in\mathbb{N}$. We also define:
\begin{equation*}
\Delta_{\infty}(G):=\sup\{{\rm diam}(\|\cdot\|_S)|\ S\subset G\text{ with }\card{S}<\infty\text{ and }\dl S\dr=G\}\in\mathbb{N}_0\cup\{+\infty\}
\end{equation*}
with $\Delta_{\infty}(G)$ defined as $-\infty$, if there is no finite, normally generating set $S\subset G$.
Also note $\Delta_k(G)\leq\Delta_{\infty}(G)$ for all $k\in\mathbb{N}$.

We will omit defining the simply-connected split Chevalley-Demazure groups $G(\Phi,R)$ and the corresponding root elements $\varepsilon_{\alpha}(x)$ in this paper. Instead, we use a representation of the complex, simply-connected Lie group ${\rm Sp}_4(\mathbb{C})$ that gives the following, classical definition of $G(C_2,R)={\rm Sp}_4(R)$ for $R$ a commutative ring with $1:$

\begin{mydef}
Let $R$ be a commutative ring with $1$ and let 
\begin{equation*}
{\rm Sp}_4(R):=\{A\in R^{4\times 4}|A^TJA=J\} 
\end{equation*}
be given with 
\begin{equation*}
J=
\left(\begin{array}{c|c}
0_2	& I_2 \\
   \midrule
   -I_2 & 0_2
\end{array}\right)
\end{equation*}
\end{mydef}

We can choose a system of positive simple roots $\{\alpha,\beta\}$ in $C_2$ such that the Dynkin-diagram of this system of positive simple roots has the following form 

\begin{center}
				\begin{tikzpicture}[
        shorten >=1pt, auto, thick,
        node distance=2.5cm,
    main node/.style={circle,draw,font=\sffamily\small\bfseries},
		 mynode/.style={rectangle,fill=white,anchor=center}
														]
      \node[main node] (1) {$\beta$};
      %\node[main node] (2) [right of=1] {$\beta$};
			\node[main node] (3) [left of=1] {$\alpha$};
			\node[mynode] (6) [left of=3] {$C_2:$};
				\path (3) edge [double,<-] node {} (1);
				%\path	(1) edge [double,->] node {a} (2);	
				
						\end{tikzpicture}
						\end{center}	

Then subject to the choice of the maximal torus in ${\rm Sp}_{4}(\mathbb{C})$ as diagonal matrices in ${\rm Sp}_{4}(\mathbb{C})$, the root elements for positive roots in $G(C_2,R)={\rm Sp}_4(R)$ can be chosen as: 
$\varepsilon_{\alpha}(t)=I_4+t(e_{1,2}-e_{4,3}), \varepsilon_{\beta}(t)=I_4+te_{2,4},\varepsilon_{\alpha+\beta}(t)=I_4+t(e_{14}+e_{23})$ and
$\varepsilon_{\alpha+\beta}(t)=I_4+te_{13}$ for all $t\in R.$ Root elements for negative roots $\phi\in C_2$ and $x\in R$ are then $\varepsilon_{\phi}(x)=\varepsilon_{-\phi}(x)^T.$ 

Next, let $\phi\in C_2.$ Then the group elements $\varepsilon_{\phi}(t)$ are \textit{additive in $t\in R$}, that is $\varepsilon_{\phi}(t+s)=\varepsilon_{\phi}(t)\varepsilon_{\phi}(s)$ holds for all $t,s\in R$. Further, a couple of commutator formulas, expressed in the next lemma, hold. We will use the additivity and the commutator formulas implicitly throughout the thesis usually without reference.

\begin{Lemma}\cite[Proposition~33.2-33.5]{MR0396773}
\label{commutator_relations}
Let $R$ be a commutative ring with $1$ and let $a,b\in R$ be given.
\begin{enumerate}
\item{If $\phi,\psi\in C_2$ are given with $0\neq\phi+\psi\notin C_2$, then $(\varepsilon_{\phi}(a),\varepsilon_{\psi}(b))=1.$}
\item{If $\alpha,\beta$ are positive, simple roots in $C_2$ with $\alpha$ short and $\beta$ long, then
\begin{align*}
&(\varepsilon_{\alpha+\beta}(b),\varepsilon_{\alpha}(a))=\varepsilon_{2\alpha+\beta}(\pm 2ab)\text{ and}\\
&(\varepsilon_{\beta}(b),\varepsilon_{\alpha}(a))=\varepsilon_{\alpha+\beta}(\pm ab)\varepsilon_{2\alpha+\beta}(\pm a^2b).
\end{align*}
}
\end{enumerate}
\end{Lemma}

We also define the Weyl group elements and diagonal elements in $G(\Phi,R)$:

\begin{mydef}
Let $R$ be a commutative ring with $1$ and let $\Phi$ be a root system. Define for $t\in R^*$ and $\phi\in\Phi$ the elements:
\begin{equation*}
w_{\phi}(t):=\varepsilon_{\phi}(t)\varepsilon_{-\phi}(-t^{-1})\varepsilon_{\phi}(t).%,h_{\phi}(t):=w_{\phi}(t)w_{\phi}(1)^{-1}.
\end{equation*}
We will often write $w_{\phi}:=w_{\phi}(1).$ We also define $h_{\phi}(t):=w_{\phi}(t)w_{\phi}(1)^{-1}$ for $t\in R^*$ and $\phi\in\Phi.$
\end{mydef} 

\begin{remark}\label{weyl_group_elements_in_the_group}
Let $\Pi=\{\alpha_1,\dots,\alpha_u\}$ be a system of simple, positive roots in the root system $\Phi$. If $w=w_{\alpha_{i_1}}\cdots w_{\alpha_{i_k}}$ is an element of the Weyl group $W(\Phi)$, then there is an element $\widetilde{w}\in G(\Phi,R)$ defined by $\widetilde{w}:=w_{\alpha_{i_1}}(1)\cdots w_{\alpha_{i_k}}(1).$ We will often denote this element $\widetilde{w}$ of $G(\Phi,R)$ by $w$ as well.
\end{remark}

%Using these Weyl group elements, we obtain:

%\begin{Lemma}\cite[Chapter~3, p.~23, Lemma~20(b)]{MR3616493}
%\label{Weyl_group_conjugation_invariance1}
%Let $R$ be a commutative ring with $1$ and $\Phi$ an irreducible root system with $\Pi$ its system of simple roots.
%Let $\phi\in\Phi,\alpha\in\Pi$ and $x\in R,t\in R^*$ be given. Then $\varepsilon_{\phi}(x)^{w_{\alpha}}=\varepsilon_{w_{\alpha}(\phi)}(\pm x)$ and $\varepsilon_{\phi}(x)^{h_{\alpha}(t)}=\varepsilon_{\phi}(t^{\langle\phi,\alpha\rangle}x)$ hold and so for each $S\subset G(\Phi,R)$, one has 
%\begin{equation*}
%\|\varepsilon_{\phi}(x)\|_S=\|\varepsilon_{w_{\alpha}(\phi)}(x)\|_S.
%\end{equation*}   
%Here the element $w_{\alpha}(\phi)$ is defined by the action of $W(\Phi)$ on $\Phi$.
%\end{Lemma}

Further, we use the following concept:

\begin{mydef}
Let $R$ be a commutative ring with $1, \Phi$ an irreducible root system, $\phi\in\Phi$ and let $S\subset G(\Phi,R)$ be given. Then for $k\in\mathbb{N}_0$ 
define the subset $\varepsilon(S,\phi,k)$ of $R$ as $\{x\in R|\ \varepsilon_{\phi}(x)\in B_S(k)\}$. Further for $A\in G(\Phi,R)$, set 
$\varepsilon(A,\phi,k):=\varepsilon(\{A\},\phi,k).$
\end{mydef}

The subgroup $U^+(\Phi,R)$, called \textit{the subgroup of upper unipotent elements of $G(\Phi,R)$}, is the subgroup of $G(\Phi,R)$ generated by the root elements $\varepsilon_{\phi}(x)$ for $x\in R$ and $\phi\in\Phi$ a positive root. Similarly, one can define $U^-(\Phi,R)$, \textit{the subgroup of lower unipotent elements} of $G(\Phi,R)$ by root elements for negative roots. We also define $B(\Phi,R):=B^+(\Phi,R):=B(R)$ as the subgroup of $G(\Phi,R)$ generated by the sets $U^+(\Phi,R)$ and $\{h_{\phi}(t)|t\in R^*,\phi\in\Phi\}.$ This subgroup is called the \textit{upper Borel subgroup of $G(\Phi,R)$.}  Similarly, one defines the \textit{lower Borel subgroup $B^-(\Phi,R)$ of $G(\Phi,R).$} 

Further for a non-trivial ideal $I\subset R$, we denote the group homomorphism $G(\Phi,R)\to G(\Phi,R/I)$ induced by the quotient map $\pi_I:R\to R/I$ by $\pi_I$ as well. This group homomorphism is commonly called the \textit{reduction homomorphism induced by $I$.}
Next, we define:

\begin{mydef}
Let $R$ be a commutative ring with $1$, $I$ an ideal in $R$, $\Phi$ an irrducible root system and $S$ a subset of $G(\Phi,R)$. Then define the following two subsets of maximal ideals in $R:$
\begin{enumerate}
\item{$V(I):=\{m\text{ maximal ideal in }R|I\subset m\}$ and}
\item{$\Pi(S):=\{ m\text{ maximal ideal of $R$}|\ \forall A\in S:\pi_m(A)\text{ central in }G(\Phi,R/m)\}$}
\end{enumerate}
\end{mydef}

\section{$2R$-pseudo-good rings and possible values for $K(C_2,2R)$}

First, we define $2R$-pseudo-good rings:

\begin{mydef}
Let $R$ be a commutative ring with $1$ such that the set of coset representatives $X$ of $2R$ in $R$ can be chosen with the following properties:
\begin{enumerate}
\item{each $x\in X-2R$ is a unit in $R,$}
\item{$0\in X$ and}
\item{if $R\neq 2R$, then $1\in X.$} 
\end{enumerate}
Then we call $R$ a \textit{$2R$-pseudo-good} ring.
\end{mydef}

\begin{remark}
If $R$ is $2R$-pseudo-good, then either $R/2R$ is a field or $2$ is a unit in $R.$ This is the case, because each element $\bar{x}$ in $R/2R-\{0\}$ can be written as 
$\bar{x}=x+2R$ for some $x\in X$ a unit. But then $\bar{x}$ is itself a unit and hence each non-zero element of $R/2R$ is a unit and so $R$ is a field. On the other hand, $R=2R$ implies that $2\in R$ is a unit. We should mention that $2R$-pseudo-goodness is our own concept named so as an homage to good rings \cite{MR2161255}.
\end{remark}

We want to point out that being $2R$-pseudo-good is equivalent to the map $R^*\to R/2R,u\mapsto u+2R$ mapping onto $R/2R-\{0\}.$
Further, define for a $2R$-pseudo-good ring $R$ with the corresponding set of coset representatives $X$, the set
\begin{align*}
B_R:=\{\varepsilon_{2\alpha+\beta}(x_1)\varepsilon_{\alpha+\beta}(x_2)\varepsilon_{\beta}(x_3)\varepsilon_{\alpha}(x_4)h_{\alpha}(t)h_{\beta}(s)
|\ t,s\in X\cap R^*,x_1,x_2,x_3,x_4\in X\}.
\end{align*} 

The goal of this section is to prove:

\begin{Proposition}
\label{pseudo_good_K(R)}
Let $R$ be a $2R$-pseudo-good ring such that $2$ is not a unit and let $J\in\mathbb{N}$ be given such that ${\rm Sp}_4(R)=(U^+(C_2,R)U^-(C_2,R))^J$ or
${\rm Sp}_4(R)=(U^-(C_2,R)U^+(C_2,R))^J$. Then $K(C_2,2R)\leq 8J+6$ holds. 
\end{Proposition}

\begin{remark}
Technically, we should assume that also $R$ is a ring of algebraic integers but this does not play a role in the proof, so we will omit this assumption. Furthermore, we will write $N$ instead of $N_{C_2,2R}$ and $Q$ instead of $Q(C_2,2R)$ in this section to simplify the notation.
\end{remark}

To show this, we define the usual Cayley word norms on finitely generated groups next:

\begin{mydef}
Let $G$ be a group and $S\subset G$ with $S=S^{-1}$ a generating set of $G$ be given. Then define the function $l_S:G\to\mathbb{N}_0$ by
$l_S(1):=0$ and by 
\begin{equation*}
l_S(x):=\min\{n\in\mathbb{N}|\exists s_1,\dots,s_n\in S:x=s_1\cdots s_n\}
\end{equation*}
for $x\neq 1.$
\end{mydef}

\begin{remark}
Note that in this definition $S$ does not normally generate $G$ but generate $G.$
\end{remark}

We also need the following definition:

\begin{mydef}
Let $G$ be a group and $S\subset G$ be given with $S=S^{-1}$ a generating set of $G$. Further, let $w=s_1\cdots s_n$ be given with all $s_i\in S$. 
\begin{enumerate}
\item{The tuple (or string) $(s_1,\dots,s_n)\in S^n$ is called an \textit{expression for $w$ in terms of $S$ of length $n$}. If $n=l_S(w)$ holds, then the tuple $(s_1,\dots,s_n)$ is called \textit{a minimal expression for $w$ (with respect to $S$).}} 
%\item{
%Let a sequence of integers $1\leq i_1<i_2<\cdots<i_k\leq n$ be given. Then
%\begin{equation*}
%((s_{i_1},i_1),(s_{i_2},i_2)\dots,(s_{i_k},i_k))
%\end{equation*} 
%is called a \textit{subexpression} of $(s_1,\dots,s_n).$
%}
\item{
An element $w'\in G$ is called a \textit{subword of $(s_1,\dots s_n)$} if there is a sequence of integers $1\leq i_1<i_2<\cdots<i_k\leq n$ such that
$w'=s_{i_1}\cdots s_{i_k}$ and $l_S(w')=k$.}
\end{enumerate}
\end{mydef}

%\begin{remark}
%\hfill
%\begin{enumerate}
%\item{We will usually omit writing down the positions when denoting subexpressions to simplify notations.
%So for example, we will write $(s_{i_1},s_{i_2}\dots,s_{i_k})$ instead of\\ 
%$((s_{i_1},i_1),(s_{i_2},i_2)\dots,(s_{i_k},i_k)).$}
%\item{
%The set $S(s_1,\dots,s_n)$ depends on the \textit{string} $(s_1,\dots,s_n)$ and not on the group element $w=s_1\cdots s_n$ represented by the string.
%Yet $S(s_1,\dots,s_n)$ is a subset of $G$ itself and not of the set $S^{<+\infty}$ of strings in $S.$} 
%\end{enumerate}
%\end{remark}

If $G$ is the Weyl group $W(\Phi)$ of an irreducible root system $\Phi$, then the generating set $S$ is usually chosen as the set 
$F=\{w_{\alpha_1},\dots,w_{\alpha_u}\}$ of fundamental reflections associated to a system of positive, simple roots $\Pi=\{\alpha_1,\dots,\alpha_u\}.$
%However, according to \cite[Chapter~8, p.~74, Lemma~53]{MR3616493}, if $(w_{\alpha_{j_1}},\dots,w_{\alpha_{j_k}})$ is a minimal expression with respect to $F$ for an element $w\in W(\Phi),$ then the set $S(w_{\alpha_{j_1}},\dots,w_{\alpha_{j_k}})$ is actually independent of the minimal expression $(w_{\alpha_{j_1}},\dots,w_{\alpha_{j_k}})$ and only depends on the element $w$ itself. Consequently, we will write $S(w)$ in this case.

The group $W(\Phi)$ further contains a unique element $w_0$ with the property that $l_F(w_0)$ is maximal, named the longest element of $W(\Phi).$
According to \cite[Appendix, p.~151, (24)Theorem]{MR3616493}, this element $w_0$ can equivalently characterized by the property that $w_0(\phi)$ is a negative root in $\Phi$ for each positive root $\phi\in\Phi.$ 
Next, note that the Weyl group $W(C_2)$ is generated by the set $F:=\{w_{\alpha},w_{\beta}\}$ for $\alpha$ and $\beta$ chosen as in the beginning of 
Section~\ref{definition}. Then one easily checks that $w_0=(w_{\alpha}\cdot w_{\beta})^2$ and $l_F(w_0)=4$ holds.

To prove Proposition~\ref{pseudo_good_K(R)}, we show the following proposition now:

\begin{Proposition}
\label{error_term_inert}
Let $R$ be a $2R$-pseudo-good ring, $w_1=s^{(1)}_1\cdots s^{(1)}_{k_1}$ and $w_2=s^{(2)}_1\cdots s^{(2)}_{k_2}$ elements of $W(C_2)$ with 
$s^{(1)}_1,\dots,s^{(1)}_{k_1},s^{(2)}_1,\dots,s^{(2)}_{k_2}$ elements of $F$ and $l_F(w_1)=k_1$ and $l_F(w_2)=k_2.$
Then up to multiplication by $l_F(w_2)$ elements of $Q$, each element of
$(B(C_2,R)w_1B(C_2,R))\cdot(B(C_2,R)w_2 B(C_2,R))$ is an element of $B(C_2,R)wB(C_2,R)$ for $w$ some subword of 
the (possibly non-minimal) expression $(s^{(1)}_1,\dots,s^{(1)}_{k_1},s^{(2)}_1,\dots,s^{(2)}_{k_2}).$
\end{Proposition}

We first show how to derive Proposition~\ref{pseudo_good_K(R)} from this proposition:

\begin{proof}
To begin, note that we may assume wlog that each $A\in N\subset{\rm Sp}_4(R)$ can be written as
\begin{equation*}
A=\prod_{i=1}^J u_i^+u_i^-
\end{equation*}
for all $u_i^+$ elements of $U^+(C_2,R)$ and all $u_i^-$ elements of $U^-(C_2,R)$.
But each element $w_0^{-1}u_i^-w_0$ is a product of root elements of positive roots in $C_2$ and hence an element of $B(C_2,R).$
But then $w_0^{-1}=-w_0$ implies
\begin{equation*}
u_i^+u_i^-=(u_i^+w_0)(w_0^{-1}u_i^-w_0)w_0^{-1}\in(B(C_2,R)w_0)\cdot(B(C_2,R)w_0)\subset(B(C_2,R)w_0B(C_2,R))^2.
\end{equation*}
holds for all $i.$ This implies $A\in (B(C_2,R)w_0B(C_2,R))^{2J}$.

But $l_F(w_0)=4$ holds, so according to Proposition~\ref{error_term_inert}, the matrix $A$ can be written as a product $b'_1wb'_2$ for $b'_1,b'_2\in B(C_2,R)$ and $w\in W(C_2)$ after multiplication by $l_F(w_0)(2J-1)\leq 4(2J-1)$ elements of $Q.$ But each element of $B(C_2,R)wB(C_2,R)$ is conjugate to an element of 
$B(C_2,R)w.$ Observe that each element of $B(C_2,R)$ has the form
\begin{equation*}
\varepsilon_{2\alpha+\beta}(t_{2\alpha+\beta})\varepsilon_{\alpha+\beta}(t_{\alpha+\beta})\varepsilon_{\beta}(t_{\beta})\varepsilon_{\alpha}(t_{\alpha})
h_{\alpha}(s_{\alpha})h_{\beta}(s_{\beta})
\end{equation*}
for $t_{2\alpha+\beta},t_{\alpha+\beta},t_{\beta},t_{\alpha}\in R$ and $s_{\alpha},s_{\beta}\in R^*.$ Hence after multiplication with $4$ elements of $Q,$ we may assume that $t_{2\alpha+\beta},t_{\alpha+\beta},t_{\beta},t_{\alpha}$ are elements of the set $X$ of coset representatives of $2R$ in $R$ given by $2R$-pseudo-goodness instead. Furthermore, 
\begin{equation*}
h_{\alpha}(s_{\alpha})=w_{\alpha}(s_{\alpha})w_{\alpha}^{-1}
\end{equation*}
holds and $w_{\alpha}(s_{\alpha})=\varepsilon_{\alpha}(s_{\alpha})\varepsilon_{-\alpha}(-s_{\alpha}^{-1})\varepsilon_{\alpha}(s_{\alpha}).$ 
Note, that all elements of $X-\{0\}$ are units in $R$ and so we can consider the set
\begin{equation*}
Y:=\{-x^{-1}|x\in X-\{0\}\}\cup\{0\}.
\end{equation*}
One easily checks that this set $Y$ is also a set of coset representatives of $2R$ in $R.$ Thus after multiplication with $3$ elements of $Q$, we may assume that 
$s_{\alpha}$ is an element of $X-\{0\}$. Similarly, we may assume after multiplication by $3$ elements of $Q$ that $s_{\beta}$ is an element of $X-\{0\}.$ 
So each element of $B(C_2,R)w$ agrees with an element of $B_Rw$ after multiplication by $4+3+3=10$ elements of $Q.$

To summarize: Up to multiplication by at most $4(2J-1)+10=8J+6$ elements of $Q,$ each element $A$ of $N$ can be rewritten as an element of 
$B_Rw$ for some $w\in W(C_2).$ Next, remember that $N={\rm ker}(\pi_{2R}:{\rm Sp}_4(R)\to {\rm Sp}_4(R/2R)).$
We are going to show that $B_Rw\cap N\neq\emptyset$ implies $w=I_4$ and $B_Rw\cap N=\{I_4\}.$
Together with $Q\subset N,$ this implies
\begin{equation*}
K(C_2,R)\leq 8J+6.
\end{equation*}

To show that $B_Rw\cap N\neq\emptyset$ implies $w=I_4$ and $B_Rw\cap N=\{I_4\}$, assume there is an $A=bw\in B_Rw\cap N$ for some $w\in W(C_2).$
Observe that $\pi_{2R}(A)=I_4$. 
But $\pi_{2R}(b)$ is an element of $B(R/2R,C_2)$ of ${\rm Sp}_4(R/2R)$. Further, slightly abusing notation,
we obtain $\pi_{2R}(w)=w$ and hence $\pi_{2R}(A)$ is an element of $B(R/2R,C_2)w.$ But $R/2R$ is a field.
Hence by the uniqueness of the Bruhat-decomposition for ${\rm Sp}_4(R/2R)$ \cite[Chapter~3, p.~26, Theorem~4']{MR3616493}, we obtain
$\pi_{2R}(b)=w=I_4.$ But according to the definition of $B_R$ and remembering that $X$ is a set of coset-representatives of $2R$ in $R$, this implies
$b\in\{h_{\alpha}(t)h_{\beta}(s)|\ t,s\in X\cap R^*\}$. So there are $t,s\in X\cap R^*$ with
\begin{equation*}
A=h_{\alpha}(t)h_{\beta}(s)
=
\left(\begin{array}{c|c}
	\begin{matrix}
		t & 0\\
		0 & st^{-1}\\
		\hline\
	0 & 0\\
	0 & 0
	\end{matrix}
	&
	\begin{matrix}
0 & 0\\
	0 & 0\\
	\hline\
		t^{-1} & 0\\
		0 & s^{-1}t
	\end{matrix}
	\end{array}\right) 
\end{equation*}
But $\pi_{2R}(A)=I_4$ and hence $t\equiv 1\text{ mod }2R.$ But $1\in X$ and so $t=1.$ Then $s=1$ follows the same way. Hence $A=I_4.$ This finishes the proof of 
Proposition~\ref{pseudo_good_K(R)}.
\end{proof}

%Further, this proof implies the following:

%\begin{Corollary}\label{CSP_2R_pseudo_goodness}
%Let $R$ be a $2R$-pseudo-good ring of S-algebraic integers with $R\neq 2R$. Then $N={\rm ker}(\pi_{2R}:{\rm Sp}_4(R)\to {\rm Sp}_4(R/2R))$ holds.
%\end{Corollary}

\begin{remark}
%Milnor's, Serre's and Bass' solution for the Congruence subgroup problem \cite[Theorem~3.6, Corollary~12.5]{MR244257} yields 
%$N={\rm ker}\left(\pi_{2R}:{\rm Sp}_4(R)\to {\rm Sp}_4(R/2R)\right)$ more generally for all rings of S-algebraic integers.
This proof can also be used to see that $Q(C_2,2R)$ does indeed generate $N_{C_2,2R}$ for $R$ a $2R$-pseudo-good ring without using the complete congruence subgroup property.
\end{remark}

To prove Proposition~\ref{error_term_inert}, we need: 

\begin{Lemma}
\label{sp_4_defect_explicit_technical_lemma1}
Let $R$ be a $2R$-pseudo-good ring. Then up to multiplication by an element of $Q$, we have 
\begin{equation*}
(B(C_2,R)w_{\alpha}B(C_2,R))\cdot(B(C_2,R)w_{\alpha}B(C_2,R))\subset B(C_2,R)\cup\left(B(C_2,R)w_{\alpha}B(C_2,R)\right).
\end{equation*}
The same holds for $\beta$ instead of $\alpha.$
\end{Lemma}

\begin{proof}
Let $b_1,b_2,b'_1,b'_2\in B(C_2,R)$ be given. Note that we may write $b_2b'_1$ as 
\begin{equation*}
b_2b'_1=\varepsilon_{\alpha}(a)u_{P-{\alpha}}h
\end{equation*}
for $a\in R$, $\varepsilon_{2\alpha+\beta}(b)\varepsilon_{\alpha+\beta}(c)\varepsilon_{\beta}(d)=u_{P-\{\alpha\}}$ for $b,c,d\in R$ and 
$h\in\{h_{\alpha}(t)h_{\beta}(s)|\ t,s\in R^*\}.$ Together with $w_{\alpha}^{-1}=-w_{\alpha}$, this implies:
\begin{align*}
b_1w_{\alpha}b_2b'_1w_{\alpha}b'_2=b_1w_{\alpha}\varepsilon_{\alpha}(a)u_{P-{\alpha}}hw_{\alpha}b'_2
=b_1\varepsilon_{-\alpha}(\pm a)w_{\alpha}[u_{P-{\alpha}}(-h)]w_{\alpha}^{-1}b'_2.
\end{align*}
Next, $w_{\alpha}[u_{P-{\alpha}}(-h)]w_{\alpha}^{-1}$ is an element of $B(C_2,R)$, because $w_{\alpha}u_{P-{\alpha}}w_{\alpha}^{-1}$ is a product of root elements associated to positive roots in $C_2$ and $w_{\alpha}(-h)w_{\alpha}^{-1}$ is an element of $\{h_{\alpha}(t)h_{\beta}(s)|\ t,s\in R^*\}$ as required. Thus 
$b_1w_{\alpha}b_2b'_1w_{\alpha}b'_2\in B(C_2,R)\varepsilon_{-\alpha}(\pm a)B(C_2,R)$ holds.

There are two possible cases now. Either $a$ is an element of $2R$, then we are done after multiplying with one element of $Q.$ 
On the other hand, if $a\notin 2R$ holds, then as $R$ is $2R$-pseudo-good, there is a unit $x\in R$ such that $a\equiv -x^{-1}\text{ mod }2R.$ Hence after multiplying with one element of $Q$, we may assume $a=-x^{-1}$ and so we obtain
\begin{align*}
\varepsilon_{-\alpha}(a)&=\varepsilon_{\alpha}(-x)(\varepsilon_{\alpha}(x)\varepsilon_{-\alpha}(-x^{-1})\varepsilon_{\alpha}(x))\varepsilon_{\alpha}(-x)\\
&=\varepsilon_{\alpha}(-x)w_{\alpha}(x)\varepsilon_{\alpha}(-x)=\varepsilon_{\alpha}(-x)h_{\alpha}(x)w_{\alpha}\varepsilon_{\alpha}(-x).
\end{align*}
But $\varepsilon_{\alpha}(-x)h_{\alpha}(x)$ and $\varepsilon_{\alpha}(-x)$ are elements of $B(C_2,R)$, so $\varepsilon_{-\alpha}(a)$ is an element of 
$B(C_2,R)w_{\alpha}B(C_2,R)$. Hence 
\begin{align*}
b_1w_{\alpha}b_2b'_1w_{\alpha}b'_2&\in B(C_2,R)\varepsilon_{-\alpha}(\pm a)B(C_2,R)\subset B(C_2,R)\cdot(B(C_2,R)w_{\alpha}B(C_2,R))\cdot B(C_2,R)\\
&=B(C_2,R)w_{\alpha}B(C_2,R)
\end{align*}
holds after multiplication with up to one element of $Q.$
\end{proof}

Next, we are going to prove the Proposition~\ref{error_term_inert}:

\begin{proof}
Slightly abusing notation, we set $T(w_1,w_2):=(s^{(1)}_1,\dots,s^{(1)}_{k_1},s^{(2)}_1,\dots,s^{(2)}_{k_2})$. We will first show by induction on $l_F(w_2)$ that 
\begin{equation*}
(B(C_2,R)w_1B(C_2,R))\cdot(B(C_2,R)w_2B(C_2,R))\subset\bigcup_{w\text{ subword of }T(w_1,w_2)}B(C_2,R)wB(C_2,R)
\end{equation*}
holds up to multiplication by $l_F(w_2)$ elements of $Q$.

For $l_F(w_2)=0$, we obtain $B(C_2,R)w_2B(C_2,R)=B(C_2,R)$ and hence the claim is obvious. 
So let $w_2\in W(C_2)$ be given with $l_F(w_2)\geq 1$ and assume without loss of generality that $w_2=w'_2w_{\alpha}$ and $l_F(w_2)=l_F(w'_2)+1.$ Then by induction hypothesis
\begin{align*}
&(B(C_2,R)w_1B(C_2,R))\cdot(B(C_2,R)w_2B(C_2,R))\\
&=(B(C_2,R)w_1B(C_2,R))\cdot(B(C_2,R)w'_2)\cdot(w_\alpha B(C_2,R))\\
&\subset\left[\bigcup_{w\text{ subword of }T(w_1,w'_2)}B(C_2,R)wB(C_2,R)\right]\cdot w_{\alpha}B(C_2,R)\\
&=\bigcup_{w\text{ subword of }T(w_1,w'_2)}\left(B(C_2,R)wB(C_2,R)\cdot w_{\alpha}B(C_2,R)\right)
\end{align*} 
holds up to multiplication by $l_F(w'_2)$ elements of $Q.$ Hence it suffices to consider the special case $w_2=w_{\alpha}$.
We distinguish two cases: First $l_F(w_1w_{\alpha})>l_F(w_1)$ and second $l_F(w_1w_{\alpha})<l_F(w_1).$ 

In the first case, it suffices to show that $w_1B(C_2,R)w_{\alpha}\subset B(C_2,R)w_1w_{\alpha}B(C_2,R).$ To see this let
\begin{equation*}
b=\varepsilon_{\alpha}(a)u_{P-\{\alpha\}}h\in B(C_2,R)
\end{equation*}
be given with $a\in R$, $\varepsilon_{2\alpha+\beta}(b)\varepsilon_{\alpha+\beta}(c)\varepsilon_{\beta}(d)=u_{P-\{\alpha\}}$ for $b,c,d\in R$ and 
$h\in\{h_{\alpha}(t)h_{\beta}(s)|\ t,s\in R^*\}.$ Note that 
\begin{equation*}
w_1\varepsilon_{\alpha}(a)w_1^{-1}=\varepsilon_{w_1(\alpha)}(\pm a). 
\end{equation*}
Yet according to \cite[Appendix, p.~151, (19)Lemma]{MR3616493}, the inequality $l_F(w_1w_{\alpha})>l_F(w_1)$ implies that the root $w_1(\alpha)$ is positive root. Thus $w_1\varepsilon_{\alpha}(a)w_1^{-1}\in B(C_2,R).$ On the other hand, similar to the proof of Lemma~\ref{sp_4_defect_explicit_technical_lemma1}, $w_{\alpha}^{-1}u_{P-\{\alpha\}}hw_{\alpha}$ is also an element of $B(C_2,R).$ Hence we obtain for all $b\in B(C_2,R)$ that
\begin{align*}
w_1bw_{\alpha}=w_1\varepsilon_{\alpha}(a)u_{P-\{\alpha\}}hw_{\alpha}=(w_1\varepsilon_{\alpha}(a)w_1^{-1})w_1w_{\alpha}
(w_{\alpha}^{-1}u_{P-\{\alpha\}}hw_{\alpha})\in B(C_2,R)w_1w_{\alpha}B(C_2,R).
\end{align*}
This finishes the proof of the first case. Note in particular that in the first case we need not multiply by an element of $Q$.

In the second case, we can write $w_1=w'_1w_{\alpha}$ for $l_F(w_1)=l_F(w'_1)+1$ and so
\begin{align*}
&(B(C_2,R)w_1B(C_2,R))\cdot (B(C_2,R)w_{\alpha}B(C_2,R))\\
&=(B(C_2,R)w'_1)\cdot(w_{\alpha}B(C_2,R))\cdot(B(C_2,R)w_{\alpha}B(C_2,R)).
\end{align*} 
But according to Lemma~\ref{sp_4_defect_explicit_technical_lemma1}, we know that up to multiplication by an element of $Q,$ we have 
\begin{equation*}
(w_{\alpha}B(C_2,R))\cdot(B(C_2,R)w_{\alpha}B(C_2,R))\subset B(C_2,R)\cup (B(C_2,R)w_{\alpha}B(C_2,R)).
\end{equation*}
Thus up to multiplication by an element of $Q$, we have 
\begin{align*}
&(B(C_2,R)w_1B(C_2,R))\cdot (B(C_2,R)w_{\alpha}B(C_2,R))\\
&\subset (B(C_2,R)w'_1)\cdot\left[B(C_2,R)\cup (B(C_2,R)w_{\alpha}B(C_2,R))\right]\\
&=(B(C_2,R)w'_1B(C_2,R))\cup(B(C_2,R)w'_1B(C_2,R)w_{\alpha}B(C_2,R)).
\end{align*} 
But according to the first case $B(C_2,R)w'_1B(C_2,R)w_{\alpha}B(C_2,R)\subset B(C_2,R)w'_1w_{\alpha}B(C_2,R)$. This finishes the second case and the proof of the proposition.
\end{proof}

\section{Rings of algebraic integers}

In this section, we will study certain examples of rings of algebraic integers $R$ to prove Theorem~\ref{fundamental_conjecture_partial_result} and Theorem~\ref{sp4_strong_boundedness_explicit}. This will require to talk about bounded generation by root elements and about $2R$-pseudo-goodness to apply Proposition~\ref{pseudo_good_K(R)}.  

\subsection{$2R$-pseudo-good rings of algebraic integers and bounded generation}

First, the definition of S-algebraic integers:

\begin{mydef}\cite[Chapter~I, \S 11]{MR1697859}\label{S-algebraic_numbers_def}
Let $K$ be a finite field extension of $\mathbb{Q}$. Then let $S$ be a finite subset of the set $V$ of all valuations of $K$ such that $S$ contains all archimedean valuations. Then the ring $\C O_S$ is defined as 
\begin{equation*}
\C O_S:=\{a\in K|\ \forall v\in V-S: v(a)\geq 0\}
\end{equation*}
and $\C O_S$ is called \textit{the ring of $S$-algebraic integers in $K.$} Rings of the form $\C O_S$ are called \textit{rings of S-algebraic integers.} 
\end{mydef}

Then for $R$ a ring of S-algebraic integers, the group ${\rm Sp}_4(R)$ is boundedly generated by root elements. The following theorem combines bounded generation results by several people among them Tavgen \cite{MR1044049}, Rapinchuk, Morgan and Sury \cite{MR3892969} as well as Carter and Keller \cite{MR704220}.

\begin{Theorem}\label{Tavgen}
Let $K$ be a number field and $R$ a ring of S-algebraic integers in $K$.
\begin{enumerate}
\item{If $R$ is a principal ideal domain, than ${\rm Sp}_{4}(R)=(U^+(C_2,R)\cdot U^-(C_2,R))^{80}$ or ${\rm Sp}_{4}(R)=(U^-(C_2,R)\cdot U^+(C_2,R))^{80}$.}
\item{If $R$ has infinitely many units, than ${\rm Sp}_{4}(R)=(U^+(C_2,R)\cdot U^-(C_2,R))^5$ or ${\rm Sp}_{4}(R)=(U^-(C_2,R)\cdot U^+(C_2,R))^5.$}
\end{enumerate}
\end{Theorem}

Using Theorem~\ref{Tavgen} and Proposition~\ref{sufficient_cond_pseudo_good2}, we can prove Theorem~\ref{fundamental_conjecture_partial_result}:

\begin{proof}
According to Theorem~\ref{Tavgen}, we can choose $J$ in Proposition~\ref{pseudo_good_K(R)} as $80$, if $R$ is a principal ideal domain and as $5,$ if $R$ has infinitely many units. Thus using Proposition~\ref{pseudo_good_K(R)}, we obtain the claim of the theorem.
\end{proof}

We will provide some examples of rings of algebraic integers that are $2R$-pseudo-good:

\begin{Proposition}\label{sufficient_cond_pseudo_good2}
Let $D$ be a square-free, positive integer with $D\equiv 5\text{ mod }8$ such that one of the two following conditions holds:
\begin{enumerate}
\item{There is a unit $u\in R_D$ with ${\rm tr}_{\mathbb{Q}[\sqrt{D}]|\mathbb{Q}}(u)$ odd.}
\item{There are odd numbers $a,b\in\mathbb{N}$ such that $b^2D=a^2\pm 4$.}
\end{enumerate} 
Then $R_D$ is $2R_D$-pseudo-good and $K(C_2,2R_D)\leq 46$.
\end{Proposition}

\begin{proof}
Due to \cite[Theorem~25]{MR3822326}, we know that $D\equiv 5\text{ mod }8$ implies that $2R_D$ is a prime ideal in $R_D.$ Thus we obtain that $R_D/2R_D$ is a field extension of $\mathbb{F}_2$ of degree two, that is $R_D/2R_D=\mathbb{F}_4.$ Further we know for each $x\in R_D$ that it is a root of
\begin{equation*}
\chi_x(T):=T^2-{\rm tr}_{\mathbb{Q}[\sqrt{D}]|\mathbb{Q}}(x)\cdot T+N_{\mathbb{Q}[\sqrt{D}]|\mathbb{Q}}(x).
\end{equation*}
The polynomial $\chi_x(T)$ is an element of $\mathbb{Z}[T].$ So $x+2R_D$ is a root of the polynomial
\begin{equation*}
\bar{\chi}_x(T):=T^2-({\rm tr}_{\mathbb{Q}[\sqrt{D}]|\mathbb{Q}}(x)+2\mathbb{Z})\cdot T+(N_{\mathbb{Q}[\sqrt{D}]|\mathbb{Q}}(x)+2\mathbb{Z})
\end{equation*}
in $\mathbb{F}_2[T].$ But if we assume that $u\in R_D$ is a unit with ${\rm tr}_{\mathbb{Q}[\sqrt{D}]|\mathbb{Q}}(u)$ odd, then this implies two things:
First, the norm $N_{\mathbb{Q}[\sqrt{D}]|\mathbb{Q}}(u)$ is either $1$ or $-1$ and so $N_{\mathbb{Q}[\sqrt{D}]|\mathbb{Q}}(x)+2\mathbb{Z}=\bar{1}.$
Second, the trace ${\rm tr}_{\mathbb{Q}[\sqrt{D}]|\mathbb{Q}}(u)$ also reduces to $\bar{1}$ in $\mathbb{F}_2.$ Thus, $u+2R_D$ is a root of the polynomial
$\bar{\chi}_x(T)=T^2+T+\bar{1}.$ But this implies that $u+2R_D$ can not be $1+2R_D$. Consequently, the group homomorphism $R_D^*\to (R_D/2R_D)^*$ can not be trivial. However, $R_D/2R_D$ is the field $\mathbb{F}_4$ and so $(R_D/2R_D)^*$ has three elements, which is a prime. Thus $R_D^*\to R_D/2R_D-\{0\}$ must be surjective and hence 
$R_D$ is indeed $2R$-pseudo-good. This finishes the proof of the first part of the proposition. For the second part, we will show that 
the fraction $x:=\frac{a+b\sqrt{D}}{2}$ is an element of $R_D$ and a unit with the property that ${\rm tr}_{\mathbb{Q}[\sqrt{D}]|\mathbb{Q}}(x)$ is odd, which will finish the proof by applying the first part. To this end, observe first that
\begin{equation*}
N_{\mathbb{Q}[\sqrt{D}]|\mathbb{Q}}(x)=\frac{a+b\sqrt{D}}{2}\cdot\frac{a-b\sqrt{D}}{2}=\frac{a^2-b^2D}{4}=\pm 1
\end{equation*} 
by assumption. Second, observe that 
\begin{equation*}
{\rm tr}_{\mathbb{Q}[\sqrt{D}]|\mathbb{Q}}(x)=\frac{a+b\sqrt{D}}{2}+\frac{a-b\sqrt{D}}{2}=a
\end{equation*}
is odd. But now remember that $x$ is a root of $\chi_x(T)\in\mathbb{Z}[T]$ and thus an element of $R_D$. Lastly, $R_D$ has infinitely many units according to Dirichlet's Unit Theorem \cite[Corollary~11.7]{MR1697859} and thus Theorem~\ref{fundamental_conjecture_partial_result} implies $K(C_2,2R_D)\leq 46$.
\end{proof}

Quite many squarefree, positive $D$ with $D\equiv 5\text{ mod }8$ satisfy the properties required in the second part of Proposition~\ref{sufficient_cond_pseudo_good2}. For example, all the ones below $100$ except for $D=37$ satisfy them, as seen by the following equations:
\begin{align*}
&1^2\cdot 5=1^2+4,1^2\cdot 13=3^2+4,1^2\cdot 21=5^2-4, 1^2\cdot 29=5^2+4, 1^2\cdot 53=7^2+4, 5^2\cdot 61=39^2+4,\\
&3^2\cdot 69=25^2-4,1^2\cdot 77=9^2-4, 1^2\cdot 85=9^2+4, 3^2\cdot 93=29^2-4
\end{align*}
However, the ring $R_{37}$ has fundamental unit $6+\sqrt{37}$ as can be seen by applying \cite[Chapter~1§7,Excercise~1]{MR1697859}. But then the image of the units $R_{37}^*$ in $(R_{37}/2R_{37})^*$ agrees with the cyclic subgroup (multiplicatively) generated by the image of $6+\sqrt{37}$ in 
$(R_{37}/2R_{37})^*.$ However, one easily checks that $6+\sqrt{37}$ maps to $1+2R_{37}$. Thus $R_{37}$ can not be $2R_{37}$-pseudo-good.
Furthermore, if $D\leq -11$ is squarefree with $D\equiv 5\text{ mod }8$, then $R_D$ can not be $2R_D$-pseudo-good either, because rings of quadratic imaginary numbers have at most one unit if $D\neq -3.$ But there are other examples of rings of algebraic integers that are $2R$-pseudo-good besides some quadratic rings of integers: 

\begin{Proposition}\label{pseudo_good_cubic_rings}
Let $p$ be a prime number greater than $2$ and consider the polynomial 
\begin{equation*}
Q_p(T):=T^3+pT^2-1\in\mathbb{Z}[T]
\end{equation*} 
Then $Q_p(T)$ is irreducible with three distinct, real roots. Let $x_p$ the biggest root of $Q_p(T)$ and let $R$ be the ring of algebraic integers in the number field 
$K:=\mathbb{Q}[x_p].$ Then $R$ is $2R$-pseudo-good and $K(C_2,2R)\leq 46$ holds.
\end{Proposition}
 
\begin{proof}
Reducing $Q_p(T)$ modulo $2$ yields the irreducible polynomial $T^3+T^2+1\in\mathbb{F}_2[T]$ and hence $Q_p(T)$ itself is irreducible as well. We leave it as an excercise to the reader to show that all roots of $Q_p(T)$ are real and different. Next, we are going to show that $R$ is $2R$-pseudo-good. To this end, we will first show that $2R$ is a prime ideal. Note that $[K:\mathbb{Q}]=3$ and hence there are the following possibilities for the prime factorization of $2R$ in $R:$
\begin{align*}
(1)2R=\C P_1\cdot\C P_2\cdot\C P_3,(2)2R=\C P_1^2\cdot\C P_2,(3)2R=\C P_1^3,
(4)2R=\C P_1\cdot\C Q,(5)2R=S
\end{align*} 
with $R/\C P_i=\mathbb{F}_2$ for $i=1,2,3$ and $\C P_1,\C P_2$ and $\C P_3$ being distinct prime ideals, $R/\C Q=\mathbb{F}_4$ and $R/S=\mathbb{F}_8.$ Observe that the following equation holds:
\begin{equation*}
1=x_p^3+px_p^2=x_p^2(x_p+p).
\end{equation*}
Thus not only is $x_p$ a unit in $R$ but also $x_p+p$ is a unit. Hence setting $\bar{x}_p:=x_p+2R$, we obtain that both $\bar{x}_p$ and $\bar{x}_p+1$ are units in the ring $R/2R.$ But all the possible $R/2R$ corresponding to the prime factorizations $(1)$ through $(4)$ have the quotient ring $R/\C P_1=\mathbb{F}_2.$ Thus $\mathbb{F}_2$ would have a unit $u$ such that $u+1$ is also a unit, but this is clearly impossible. Hence $2R$ must be a prime ideal itself. But then 
$R/2R$ must be the field $\mathbb{F}_8.$ But the unit group $\mathbb{F}_8^*$ has order $7,$ which is a prime number. Thus to prove the $2R$-pseudo-goodness of $R,$ it suffices to show that there is a unit in $R$ which maps to a non-trivial unit of $R/2R.$ However, observe that as $1=\bar{x}_p^2(\bar{x}_p+1)$ holds in $R/2R,$ it is impossible that $\bar{x}_p$ is equal to $1.$ Thus $R$ is a $2R$-pseudo-good ring. The inequality $K(C_2,2R)\leq 46$ follows from Theorem~\ref{fundamental_conjecture_partial_result}, because $R$ has infinitely many units according to Dirichlet's Unit Theorem \cite[Corollary~11.7]{MR1697859}.
\end{proof}

\section{Proving strong boundedness for ${\rm Sp}_4(R)$ explicitly}

To prove Theorem~\ref{sp4_strong_boundedness_explicit}, we will rephrase \cite[Theorem~3.1]{General_strong_bound}. To this end, we note two statements:

\begin{Lemma}\cite[Lemma~3.4]{General_strong_bound}
\label{congruence_fin}
Let $R$ be a commutative ring with $1$ such that $(R:2R)<\infty$ and such that $G:={\rm Sp}_4(R)$ is boundedly generated by root elements. 
Further define 
\begin{equation*}
Q':=\{\varepsilon_{\phi}(2x)|\ x\in R,\phi\in C_2\}.
\end{equation*}
and $N':=\dl Q'\dr$ and let $\|\cdot\|_{Q'}:N'\to\mathbb{N}_0$ be the conjugation invariant word norm on $N'$ defined by $Q'.$  
\begin{enumerate}
\item{
Then the group $G/N'$ is finite.}
\item{Then there is a minimal $K'(C_2,2R)\in\mathbb{N}$ such that $\|N'\|_{Q'}\leq K'(C_2,2R)$.}
\end{enumerate}
\end{Lemma}

\begin{remark}
If $R$ is a ring of algebraic integers, then the groups $N_{C_2,2R}$ and $N'$ as well as the constants $K(C_2,2R)$ and $K'(C_2,2R)$ are the same due to the congruence subgroup property \cite[Theorem~3.6, Corollary~12.5]{MR244257}. 
\end{remark}

We also need:

\begin{Theorem}\cite[Theorem~3.2]{General_strong_bound}
\label{fundamental_reduction}
Let $R$ be a commutative ring with $1$. Then there is a constant $L(C_2,R)\in\mathbb{N}$ such that for $A\in {\rm Sp}_4(R),$ there is an ideal 
$I(A)\subset R$ with the following two properties: 
\begin{enumerate}
\item{$V(I(A))\subset\Pi(\{A\})$ and}
\item{$2I(A)\subset\varepsilon(A,\phi,L(C_2,R))$ for all $\phi\in C_2.$}
\end{enumerate}
\end{Theorem}

We can recast \cite[Theorem~3.1]{General_strong_bound} slightly more explicitly as follows:

\begin{Theorem}\label{sp_4_strong_boundedness_explicit_er}
Let $R$ be a commutative ring with $1$ and $(R:2R)<+\infty$ such that ${\rm Sp}_4(R)$ is boundedly generated by root elements. Further, let 
$L(C_2,R)\in\mathbb{N}$ be as in Theorem~\ref{fundamental_reduction} and $K'(C_2,2R)\in\mathbb{N}$ as in Lemma~\ref{congruence_fin}. Then for all $k\in\mathbb{N},$ one has
\begin{equation*}
\Delta_k({\rm Sp}_4(R))\leq L(C_2,R)\cdot K'(C_2,2R)\cdot k+\Delta_{\infty}({\rm Sp}_4(R)/N')
\end{equation*} 
\end{Theorem}

\begin{remark}
The group ${\rm Sp}_4(R)/N'$ is finite as observed in Lemma~\ref{congruence_fin} and hence $\Delta_{\infty}({\rm Sp}_4(R)/N')$ is a well-defined natural number.
\end{remark}

So to prove Theorem~\ref{sp4_strong_boundedness_explicit}, we have to determine $L(C_2,R),K'(C_2,2R)=K(C_2,2R)$ and $\Delta_{\infty}({\rm Sp}_4(R)/N')=\Delta_{\infty}({\rm Sp}_4(R)/N_{C_2,2R})=\Delta_{\infty}({\rm Sp}_4(R/2R))$ for the rings of S-algebraic integers mentioned in Theorem~\ref{sp4_strong_boundedness_explicit}. First, regarding $K(C_2,2R):$ 
Note that $\mathbb{Z}$ and $\mathbb{Z}[\frac{1+\sqrt{-3}}{2}]$ are principal ideal domains. This is obvious for $\mathbb{Z}$ and follows for $\mathbb{Z}[\frac{1+\sqrt{-3}}{2}]$, because $\mathbb{Z}[\frac{1+\sqrt{-3}}{2}]$ is a euclidean domain by way of using the norm map 
\begin{equation*}
N_{\mathbb{Q}[\sqrt{-3}]|\mathbb{Q}}:\mathbb{Z}\left[\frac{1+\sqrt{-3}}{2}\right]\to\mathbb{Z}.
\end{equation*}
Further, $R_D$ and $\mathbb{Z}[p^{-1}]$ have infinitely many units. Lastly, all the rings $\mathbb{Z},\mathbb{Z}[p^{-1}],R_D$ and $\mathbb{Z}[\frac{1+\sqrt{-3}}{2}]$ are $2R$-pseudo-good: This is clear for $\mathbb{Z},\mathbb{Z}[p^{-1}]$ and follows for $R_D$ from Proposition~\ref{sufficient_cond_pseudo_good2}. For $R:=\mathbb{Z}[\frac{1+\sqrt{-3}}{2}]$ this follows from $X:=\{0,1,\frac{1+\sqrt{-3}}{2},\frac{1-\sqrt{-3}}{2}\}$ being a set of coset representatives of $2R$ in $R$ satisfying the definition of $2R$-pseudo-goodness. Thus applying Theorem~\ref{fundamental_conjecture_partial_result}, we obtain
\begin{align*}
K(C_2,2\mathbb{Z})\leq 646,K(C_2,2\mathbb{Z}[\frac{1+\sqrt{-3}}{2}])\leq 646,K(C_2,2\mathbb{Z}[p^{-1}])\leq 46\text{ and }K(C_2,2R_D)\leq 46.
\end{align*}

Next, we determine $L(C_2,R)$ and $\Delta_{\infty}({\rm Sp}_4(R/2R)).$ In regards to $L(C_2,R)$, we state the following:

\begin{Theorem}\label{B2_centralization_explicit}
Let $R$ be a principal ideal domain. Then $L(C_2,R)$ as in Theorem~\ref{fundamental_reduction} can be chosen as $320.$
\end{Theorem}

We will omit the proof as it is rather lengthy and similar to the proofs of \cite[Theorem~2.3]{explicit_strong_bound_sp_2n} and \cite[Proposition~6.17]{KLM}. The interested reader however can find the proof in the author's PhD thesis \cite[Theorem~4.2.1]{PhD_thesis}. We note that all the rings $R$ from Theorem~\ref{sp4_strong_boundedness_explicit} are principal ideal domains. This is true by assumption for $R_D$, obvious for $\mathbb{Z}$ and $\mathbb{Z}[p^{-1}]$ and we saw already that $\mathbb{Z}[\frac{1+\sqrt{-3}}{2}]$ is euclidean. Thus for all rings $R$ as in Theorem~\ref{sp4_strong_boundedness_explicit} the constant $L(C_2,R)$ can be chosen as $320.$

Next, we must determine $\Delta_{\infty}({\rm Sp}_4(R/2R))$ for these rings $R$:    

\begin{Proposition}\label{delta_infty_sp_4_k} 
Let $K$ be either $\mathbb{F}_2$ or $\mathbb{F}_4.$
\begin{enumerate}
\item{If $K=\mathbb{F}_2$, then $\Delta_{\infty}({\rm Sp}_4(K))\leq 5$.}
\item{If $K=\mathbb{F}_4$, then $\Delta_{\infty}({\rm Sp}_4(K))\leq 4$.}
\end{enumerate} 
\end{Proposition}

\begin{proof}
For the purposes of this proof, we define the following invariant for a group $G$:
\begin{align*}
{\rm cn}(G):=\min\{n\in\mathbb{N}|\forall C\text{ a  conjugacy class in }G\text{ with }G=\langle C\rangle:G=C^n\}
\end{align*}
with ${\rm cn}(G)$ being defined as $+\infty$, if the corresponding set is empty. Then one easily obtains $\Delta_1(G)\leq{\rm cn}(G)$ for all groups $G.$ Next, let $S$ be a normally generating subset of ${\rm Sp}_4(K).$ We distinguish the two cases for $K$. First, assume $K=\mathbb{F}_4.$ Observe that ${\rm Sp}_4(K)={\rm PSp}_4(K)$, because each scalar matrix in ${\rm Sp}_4(K)$ must be $I_4$ as ${\rm char}(K)=2.$ 
But as $K\neq\mathbb{F}_2$, the group ${\rm Sp}_4(K)={\rm PSp}_4(K)$ is simple by \cite[Chapter~4, p.~33, Theorem~5]{MR3616493}. Hence any non-trivial element 
$A_S\in S$ also normally generates ${\rm Sp}_4(K).$ 
But in the second case $K=\mathbb{F}_2$, it is well-known that the group ${\rm Sp}_4(K)$ is isomorphic to the permutation group $S_6.$ This group however only has three normal subgroups namely $\{1\},S_6$ and the alternating subgroup $A_6.$ Thus if we pick an $A_S\in S$, that does not lie in $A_6,$ then necessarily $A_S$ normally generates ${\rm Sp}_4(K).$ So for each normally generating set $S$ of ${\rm Sp}_4(K)$ there is an $A_S\in S$ that normally generates ${\rm Sp}_4(K)$ for both $K=\mathbb{F}_2$ and $K=\mathbb{F}_4.$ This implies:
\begin{align*}
\Delta_{\infty}({\rm Sp}_4(K))&\geq\Delta_1({\rm Sp}_4(K))\geq\sup\{\|{\rm Sp}_4(K)\|_{A_S}|\ S\text{ normally generates }{\rm Sp}_4(K)\}\\
&\geq\sup\{\|{\rm Sp}_4(K)\|_S|\ S\text{ normally generates }{\rm Sp}_4(K)\}=\Delta_{\infty}({\rm Sp}_4(K)).
\end{align*}  
Hence to give upper bounds on $\Delta_{\infty}({\rm Sp}_4(K))$, it suffices to give upper bounds on ${\rm cn}({\rm Sp}_4(K)).$
First, for ${\rm Sp}_4(\mathbb{F}_2)=S_6,$ the invariant ${\rm cn}(S_6)$ can be determined to be $5$ from the main result in \cite{brenner1978covering}.
Second for $K=\mathbb{F}_4,$ we use that the paper \cite{Karni_paper} contains a list of the invariants ${\rm cn}(G)$ for simple groups $G$ with less than $1000000$ elements calculated using a computer algebra system and states on page $61$ that ${\rm cn}({\rm Sp}_4(\mathbb{F}_4))=4.$ This yields $\Delta_{\infty}({\rm Sp}_4(\mathbb{F}_4))\leq 4$.  
\end{proof}

Note next that $\mathbb{Z}/2\mathbb{Z}=\mathbb{F}_2=\mathbb{Z}[p^{-1}]/2\mathbb{Z}[p^{-1}]$ and 
$R_D/2R_D=\mathbb{F}_4=\mathbb{Z}[\frac{1+\sqrt{-3}}{2}]/2\mathbb{Z}[\frac{1+\sqrt{-3}}{2}]$ hold. Hence Proposition~\ref{delta_infty_sp_4_k} implies
\begin{align*}
&\Delta_{\infty}({\rm Sp}_4(\mathbb{Z})/N')\leq 5\geq\Delta_{\infty}({\rm Sp}_4(\mathbb{Z}[p^{-1}])/N')\\
&\Delta_{\infty}({\rm Sp}_4(R_D)/N')\leq 4\geq\Delta_{\infty}({\rm Sp}_4(\mathbb{Z}[\frac{1+\sqrt{-3}}{2}])/N')
\end{align*}
Combining these bounds on $\Delta_{\infty}({\rm Sp}_4(R)/N')$, the value of $L(C_2,R)=320$ from Theorem~\ref{B2_centralization_explicit} and the bounds on
$K(C_2,2R)$ determined before with Theorem~\ref{sp_4_strong_boundedness_explicit_er}, we obtain Theorem~\ref{sp4_strong_boundedness_explicit}. 

\section{Classifying normal generating subsets of ${\rm Sp}_4(R)$}

In this section, we show Theorem~\ref{classifying_normal_generating_subsets_better}. To prove this, we need the following:

\begin{Lemma}\label{classifying_normal_generating_subsets_better_prep_lemma}
Let $R$ be a commutative ring with $1$ and of characteristic $2$ and let $S\subset{\rm Sp}_4(R)$ be given with $\Pi(S)=\emptyset.$ Then 
$\varepsilon_{2\alpha+\beta}(1)\varepsilon_{\alpha+\beta}(1)\in\dl S\dr$ and $\nu_2(R):=(x^2-x|x\in R)\subset\{y\in R|\forall\phi\in C_2:\varepsilon_{\phi}(x)\in\dl S\dr\}$ hold. 
\end{Lemma}

\begin{proof}
We will first show that $\varepsilon_{2\alpha+\beta}(1)\varepsilon_{\alpha+\beta}(1)\in\dl S\dr$. Second, we will show that any normal subgroup of ${\rm Sp}_4(R)$ containing $\varepsilon_{2\alpha+\beta}(1)\varepsilon_{\alpha+\beta}(1)$ also contains $\varepsilon_{\phi}(y)$ for any $\phi\in C_2$ and $y\in\nu_2(R).$
For the first claim note that $\Pi(S)=\emptyset$ implies 
\begin{equation*}
R=\sum_{A\in S}(x_{ij}^{(A)},x_{ii}^{(A)}-x_{jj}^{(A)}|1\leq i\neq j\leq 4)
\end{equation*}
for $A=(x_{ij}^{(A)}).$
But using Freshmen's dream, this implies that there are $A_1,\dots,A_k\in S$ as well as $u_{ij}^{(l)},t_{ij}^{(l)}\in R$ with
\begin{equation*}
1=\sum_{l=1}^k\sum_{1\leq i\neq j\leq 4} (u_{ij}^{(l)}x_{ij}^{(A_l)}+t_{ij}^{(l)}(x_{ii}^{(A_l)}-x_{jj}^{(A_l)}))^2.
\end{equation*}
However using \cite[Theorem~2.6, 4.2, 5.1, 5.2]{MR1162432}, we know that the element 
\begin{equation*}
\varepsilon_{2\alpha+\beta}((u_{ij}^{(l)}x_{ij}^{(A_l)}+t_{ij}^{(l)}(x_{ii}^{(A_l)}-x_{jj}^{(A_l)}))^2)\cdot
\varepsilon_{\alpha+\beta}((u_{ij}^{(l)}x_{ij}^{(A_l)}+t_{ij}^{(l)}(x_{ii}^{(A_l)}-x_{jj}^{(A_l)}))^2)
\end{equation*}
is contained in the normal subgroup generated by $A_l$ for all $1\leq i\neq j\leq 4$ and all $1\leq l\leq k.$
But the elements of the root subgroups $\varepsilon_{2\alpha+\beta}(R)$ and $\varepsilon_{\alpha+\beta}(R)$ commute and this implies that the normal subgroup generated by $S$ contains the product
\begin{align*}
&\prod_{1\leq l\leq k}\prod_{1\leq i\neq j\leq 4}\varepsilon_{2\alpha+\beta}((u_{ij}^{(l)}x_{ij}^{(A_l)}+t_{ij}^{(l)}(x_{ii}^{(A_l)}-x_{jj}^{(A_l)}))^2)\cdot
\varepsilon_{\alpha+\beta}((u_{ij}^{(l)}x_{ij}^{(A_l)}+t_{ij}^{(l)}(x_{ii}^{(A_l)}-x_{jj}^{(A_l)}))^2)\\
&=\varepsilon_{2\alpha+\beta}(\sum_{l=1}^k\sum_{1\leq i\neq j\leq 4} (u_{ij}^{(l)}x_{ij}^{(A_l)}+t_{ij}^{(l)}(x_{ii}^{(A_l)}-x_{jj}^{(A_l)}))^2)
\cdot\varepsilon_{\alpha+\beta}(\sum_{l=1}^k\sum_{1\leq i\neq j\leq 4} (u_{ij}^{(l)}x_{ij}^{(A_l)}+t_{ij}^{(l)}(x_{ii}^{(A_l)}-x_{jj}^{(A_l)}))^2)\\
&=\varepsilon_{2\alpha+\beta}(1)\varepsilon_{\alpha+\beta}(1)
\end{align*}
This finishes the first step. For the second step, let $M$ be the normal subgroup of ${\rm Sp}_4(R)$ generated by 
$\varepsilon_{2\alpha+\beta}(1)\varepsilon_{\alpha+\beta}(1).$ Then we obtain for $x\in R$ that
\begin{equation}
\label{char2_uniform_boundedness_eq0}
M\ni(\varepsilon_{\alpha+\beta}(1)\varepsilon_{2\alpha+\beta}(1),\varepsilon_{-\beta}(x))=\varepsilon_{\alpha}(x)\varepsilon_{2\alpha+\beta}(x).
\end{equation}
This yields that
\begin{align*}
M\ni &w_{\beta}w_{\alpha}w_{\beta}\varepsilon_{\alpha}(x)\varepsilon_{2\alpha+\beta}(x)w_{\beta}^{-1}w_{\alpha}^{-1}w_{\beta}^{-1}\\
&=w_{\beta}w_{\alpha}\varepsilon_{\alpha+\beta}(x)\varepsilon_{2\alpha+\beta}(x)w_{\alpha}^{-1}w_{\beta}^{-1}\\
&=w_{\beta}\varepsilon_{\alpha+\beta}(x)\varepsilon_{\beta}(x)w_{\beta}^{-1}=\varepsilon_{\alpha}(x)\varepsilon_{-\beta}(x).
\end{align*}
On the other hand (\ref{char2_uniform_boundedness_eq0}) implies for $x,y\in R$ that 
\begin{align*}
M\ni(\varepsilon_{\alpha}(y)\varepsilon_{2\alpha+\beta}(y),\varepsilon_{-(\alpha+\beta)}(x))
&={^{\varepsilon_{\alpha}(y)}(\varepsilon_{2\alpha+\beta}(y),\varepsilon_{-(\alpha+\beta)}(x))}\varepsilon_{-\beta}(2xy)\\
&=\varepsilon_{\alpha}(xy)\varepsilon_{-\beta}(x^2y).
\end{align*}
But this implies that 
\begin{equation*}
\varepsilon_{-\beta}(x^2y-xy)=\varepsilon_{-\beta}(x^2y)\varepsilon_{-\beta}(xy)
=\left(\varepsilon_{-\beta}(x^2y)\varepsilon_{\alpha}(xy)\right)\cdot\left(\varepsilon_{\alpha}(xy)\varepsilon_{-\beta}(xy)\right)\in M.
\end{equation*}
This implies in particular that the ideal $\nu_2(R):=(x^2-x|x\in R)$ is contained in $\{z\in R|\forall\phi\in C_2\text{ long: }\varepsilon_{\phi}(z)\in M\}.$ However, \cite[Lemma~4.8]{General_strong_bound} implies that if $\varepsilon_{\phi}(z)\in M$ for all $\phi\in C_2$ long, then $\varepsilon_{\phi}(z)\in M$ for all $\phi\in C_2$. This finishes this proof.
\end{proof}

We can prove Theorem~\ref{classifying_normal_generating_subsets_better} now:

\begin{proof}
According to \cite[Corollary~3.11]{General_strong_bound}, the set $S$ normally generates ${\rm Sp}_4(R)$ precisely if $\Pi(S)=\emptyset$ and the image of $S$ normally generates the quotient ${\rm Sp}_4(R)/N'$ for
\begin{equation*}
N'=\langle\langle\varepsilon_{\phi}(2x)|x\in R,\phi\in C_2\rangle\rangle.
\end{equation*} 
As before, the congruence subgroup property \cite[Theorem~3.6, Corollary~12.5]{MR244257} can be used to identify the normal subgroup $N'$ as the kernel of the reduction homomorphism $\pi_{2R}:{\rm Sp}_4(R)\to{\rm Sp}_4(R/2R)$ and hence ${\rm Sp}_4(R)/N'={\rm Sp}_4(R/2R).$ 
So $S$ normally generates ${\rm Sp}_4(R)$ precisely if $\Pi(S)=\emptyset$ and $S$ maps to a normal generating set of the finite group ${\rm Sp}_4(R/2R).$
So to finish the proof of this corollary, it suffices to show that $\Pi(S)=\emptyset$ and $S$ mapping to a generating set of the abelianization of ${\rm Sp}_4(R)$, implies that ${\rm Sp}_4(R/2R)$ is normally generated by the image $\bar{S}$ of $S.$ But $R/2R$ is a finite ring and so is semi-local. Thus \cite[Corollary~2.4]{MR439947} implies that ${\rm Sp}_4(R/2R)$ is generated by root elements. Hence it suffices to show that the normal subgroup $\bar{W}$ of ${\rm Sp}_4(R/2R)$ normally generated by $\bar{S}$ contains all root elements. However, \cite[Lemma~4.8]{General_strong_bound} implies that if 
$\{\varepsilon_{\alpha}(\bar{x})|\bar{x}\in R/2R\}$ is a subset of $\bar{W}$, than $\bar{W}$ contains all root elements. Hence it suffices to show that
\begin{equation*}
\bar{R}_0:=\{\bar{x}\in R/2R|\varepsilon_{\alpha}(\bar{x})\in\bar{W}\}=R/2R.
\end{equation*} 

To this end, observe that $\Pi(S)=\emptyset$ implies $\Pi(\bar{S})=\emptyset.$ But then Lemma~\ref{classifying_normal_generating_subsets_better_prep_lemma} implies that $\nu_2(R/2R):=(\bar{x}^2-\bar{x}|\bar{x}\in R/2R)$ is a subset of $\bar{R}_0$ and $\varepsilon_{\alpha+\beta}(1)\varepsilon_{2\alpha+\beta}(1)$ is an element of $\bar{W}$. Next, let the ideal $2R$ in $R$ split into primes as follows:
\begin{equation*}
2R=\left(\prod_{i=1}^r\C P_i^{l_i}\right)\cdot\left(\prod_{j=1}^s\C Q_j^{k_j}\right)
\end{equation*}
with $[R/\C P_i:\mathbb{F}_2]=1$ for $1\leq i\leq r$ and $[R/\C Q_j:\mathbb{F}_2]>1$ for $1\leq j\leq s.$ Using the Chinese Remainder Theorem, this implies that
\begin{equation*}
R/2R=(\prod_{i=1}^r R/\C P_i^{l_i})\times(\prod_{j=1}^s R/\C Q_j^{k_j}).
\end{equation*}
As the next step, we show that $\{(0+\C P_1^{l_1},\dots,0+\C P_r^{l_r},x+\C Q_1^{k_1},\dots,x+\C Q_s^{k_s})|x\in R\}\subset\bar{R}_0.$ To this end for each $j=1,\dots,s$, pick an element $x_j\in R$ such that neither $x_j$ nor $x_j-1$ are elements of $\C Q_j.$ Such elements do exist, because otherwise 
$[R/\C Q_j:\mathbb{F}_2]=1$ would hold. This implies that 
\begin{equation*}
(0+\C P_1^{l_1},\dots,0+\C P_r^{l_r},x_1(x_1-1)+\C Q_1^{k_1},\dots,x_s(x_s-1)+\C Q_s^{k_s})
\end{equation*}
is not only an element of $\nu_2(R/2R)\subset\bar{R}_0,$ but also that $(x_1(x_1-1)+\C Q_1^{k_1},\dots,x_s(x_s-1)+\C Q_s^{k_s})$ is a unit in the quotient ring 
$\prod_{j=1}^s R/\C Q_j^{k_j}.$ But then \cite[Proposition~6.4]{General_strong_bound} implies 
\begin{equation*}
\{(0+\C P_1^{l_1},\dots,0+\C P_r^{l_r},x+\C Q_1^{k_1}\cdots\C Q_s^{k_s})|x\in R\}\subset\bar{R}_0.
\end{equation*}
Next, we show that $S$ mapping to a generating set of the abelianization of ${\rm Sp}_4(R)$, implies further that 
\begin{equation*}
\{(x+\C P_1^{l_1}\cdots\C P_r^{l_r},0+\C Q_1^{k_1}\cdots\C Q_s^{k_s})|x\in R\}\subset\bar{R}_0.
\end{equation*}
First, observe that according to the proof of \cite[Theorem~6.3]{General_strong_bound}, the abelianization homomorphism of ${\rm Sp}_4(R)$ is uniquely defined by
$A:{\rm Sp}_4(R)\to\mathbb{F}_2^r,\varepsilon_{\phi}(x)\mapsto(x+\C P_1,\dots,x+\C P_r)$ for all $\phi\in C_2.$ According to the same proof the abelianization of ${\rm Sp}_4(R)$ factors through ${\rm Sp}_4(R/2R)$. Also the element $\varepsilon_{2\alpha+\beta}(1)\varepsilon_{\alpha+\beta}(1)$ of ${\rm Sp}_4(R/2R)$ clearly vanishes under the abelianization map. Thus the abelianization map of ${\rm Sp}_4(R)$ factors through ${\rm Sp}_4(R/2R)/\bar{M}$ for $\bar{M}$ the normal subgroup of ${\rm Sp}_4(R/2R)$ generated by $\varepsilon_{2\alpha+\beta}(1)\varepsilon_{\alpha+\beta}(1).$

Second, we leave it as an exercise to the reader to show that $\bar{M}$ also contains the set
\begin{equation*}
\{\varepsilon_{\phi_1}(\bar{x})\varepsilon_{\phi_2}(\bar{x})|\bar{x}\in R/2R,\phi_1,\phi_2\in C_2\}.
\end{equation*} 
But $\bar{M}$ is contained in $\bar{W}$ and so if $\phi_1,\dots,\phi_n\in C_2$ and $x_1,\dots,x_n\in R$ are given with
\begin{equation*}
\bar{X}=\prod_{i=1}^n\varepsilon_{\phi_i}(x_i+2R),
\end{equation*}
an element of $\bar{W}$, then $\varepsilon_{\alpha}(x_1+\cdots+x_n+2R)$ is also an element of $\bar{W}$ and both $\bar{X}$ and $\varepsilon_{\alpha}(x_1+\cdots+x_n+2R)$ map to $(x_1+\cdots+x_n+\C P_1,\dots,x_1+\cdots+x_n+\C P_r)$ under the abelianization map. 

Third, $S$ maps to a generating set of the abelianization of ${\rm Sp}_4(R)$ and hence by the previous observation, there must be a $y_1\in R$ such that 
$(y_1+\C P_1,\dots,y_1+\C P_r)=(1,0,\dots,0)$ and
\begin{equation*}
(y_1+\C P_1^{l_1},\dots,y_1+\C P_r^{l_r},y_1+\C Q_1^{k_1},\dots,y_1+\C Q_s^{k_s})\in\bar{R}_0.
\end{equation*}
But we already know that $(0+\C P_1^{l_1},\dots,0+\C P_r^{l_r},y_1+\C Q_1^{k_1},\dots,y_1+\C Q_s^{k_s})\in\bar{R}_0$ and so we obtain that $(y_1+\C P_1^{l_1},y_1+\C P_2^{l_2}\cdots\C P_r^{l_r},0+\C Q_1^{k_1}\cdots\C Q_s^{l_s})$ is an element of $\bar{R}_0$. Next, observe that if $z+2R$ is an element of $\bar{R}_0$, then so is $u\cdot(z+2R)$ for any unit in $R/2R.$ This can be seen by conjugating $\varepsilon_{\alpha}(z+2R)$ with $h_{\beta}(u^{-1}).$ But $y_1$ is an element of $\C P_2\cdots\C P_r$ and hence 
\begin{equation*}
(1+\C P_1^{l_1},y_1-1+\C P_2^{l_2}\cdots\C P_r^{l_r},1+\C Q_1^{k_1}\cdots\C Q_s^{k_s})
\end{equation*}
is a unit in $R/2R$ and we already know by Lemma~\ref{classifying_normal_generating_subsets_better_prep_lemma} that 
\begin{equation*}
(0+\C P_1^{l_1},(1-y_1)y_1+\C P_2^{l_2}\cdots\C P_r^{l_r},0+\C Q_1^{k_1}\cdots\C Q_s^{k_s})
\end{equation*}
is an element of $\bar{R}_0$. So, we can conclude that
\begin{equation*}
(0+\C P_1^{l_1},y_1+\C P_2^{l_2}\cdots\C P_r^{l_r},0+\C Q_1^{k_1}\cdots\C Q_s^{k_s})
\end{equation*}
is an element of $\bar{R}_0.$ Thus finally $(y_1+\C P_1^{l_1},0+\C P_2^{l_2}\cdots\C P_r^{l_r}\cdot\C Q_1^{k_1}\cdots\C Q_s^{l_s})$
is an element of $\bar{R}_0.$ But $y_1+\C P_1^{l_1}$ is a unit in $R/\C P_1^{l_1}$ by choice and hence for any unit $u\in R/\C P_1^{l_1}$ the element 
$(u+\C P_1^{l_1},0+\C P_2^{l_2}\cdots\C P_r^{l_r}\cdot\C Q_1^{k_1}\cdots\C Q_s^{l_s})$ is an element of $\bar{R}_0.$ But each element of $R/\C P_1^{l_1}$ is a sum of at most two units according to \cite[Proposition~6.4]{General_strong_bound} and hence $(x+\C P_1^{l_1},0+\C P_2^{l_2}\cdots\C P_r^{l_r}\cdot\C Q_1^{k_1}\cdots\C Q_s^{l_s})$ is an element of $\bar{R}_0$ for all $x\in R.$ But running through similar arguments for the other prime ideals $\C P_2,\dots,\C P_r$, then finally yields that
\begin{equation*}
\{(x+\C P_1^{l_1}\cdots\C P_r^{l_r},0+\C Q_1^{k_1}\cdots\C Q_s^{l_s})|x\in R\}\subset\bar{R}_0.
\end{equation*} 
So summarizing, we finally obtain that indeed $\bar{R}_0=R/2R$ and this finishes the proof.
\end{proof}

\section{Lower bounds for $\Delta_k({\rm Sp}_4(R))$ always depend on the number $r(R)$}

In this section, we prove Theorem~\ref{lower_bounds_better}. The proof is quite similar to the proof of \cite[Theorem~2,3]{explicit_strong_bound_sp_2n}, so we will be rather brief about it. First, we need the following:

\begin{Proposition}
\label{dimension_counting_sln_sp_2n}
Let $K$ be a field, $t\in K-\{0\}$. Then the element $E:=\varepsilon_{\beta}(t)$ normally generates ${\rm Sp}_4(K)$ and 
$\|{\rm Sp}_4(K)\|_E\geq 4$. Further, $\|{\rm Sp}_4(K)\|_E\geq 5$ holds, if $K=\mathbb{F}_2.$
\end{Proposition}

\begin{proof}
The first claim is the content of \cite[Proposition~5.1]{explicit_strong_bound_sp_2n}. The second claim can be seen by noting that there is an isomorphism between ${\rm Sp}_4(\mathbb{F}_2)$ and $S_6$ that maps $\varepsilon_{\beta}(1)$ to the transposition $(4,6).$ Hence it suffices to show that $S_6$ is normally generated by $(4,6)$ and $\|S_6\|_{(4,6)}\geq 5.$ The first claim is well-known and to see the the second one observe that for $\sigma\in S_6$ the number of orbits of the induced group action of the cyclic subgroup $\langle\sigma\rangle$ on $\{1,\dots,6\}$ only depends on the conjugacy class of $\sigma$ in $S_6$ and not on the permutation $\sigma$ itself. However, for $k\in\{1,\dots,5\}$, a product of $k$ transpositions in $S_6$ has at least $6-k$ such orbits in $\{1,\dots,6\}.$ Thus the cyclce $(1,2,3,4,5,6)$, which gives rise to just one such orbit, cannot be written as a product of at most $4$ transpositions and hence $\|S_6\|_{(4,6)}\geq 5.$. 
\end{proof}

This proposition can be used now to prove Theorem~\ref{lower_bounds_better}:

\begin{proof}
We can assume without loss that $R\neq 2R.$ Let $2R$ split into distinct prime ideals as follows:
\begin{equation*}
2R=\left(\prod_{i=1}^r\C P_i^{l_i}\right)\cdot\left(\prod_{j=1}^s\C Q_j^{k_j}\right)
\end{equation*}
with $[R/\C P_i:\mathbb{F}_2]=1$ for $1\leq i\leq r$ and $[R/\C Q_j:\mathbb{F}_2]>1$ for $1\leq j\leq s.$ Next, let $c$ be the class number of $R$. Pick elements $x_1,\dots,x_r\in R$ such that $\C P_i^c=(x_i)$ for all $i.$ Also choose $r+1$ distinct prime ideals $V_{r+1},\dots,V_k$ in $R$ different from all the $\C P_1,\dots,\C P_r,\C Q_1,\dots,\C Q_s.$ Passing to the powers $V_{r+1}^c,\dots,V_k^c,$ we can find elements $v_{r+1},\dots,v_k\in R$ with $V_{r+1}^c=(v_{r+1}),\dots,V_k^c=(v_k).$ 
Further, define the following elements for $1\leq u\leq r,$ 
\begin{equation*}
r_u:=\left(\prod_{1\leq i\neq u\leq r}x_i\right)\cdot v_{r+1}\cdots v_k.
\end{equation*}
For $k\geq u\geq r+1$ set
\begin{equation*}
r_u:=x_1\cdots x_r\cdot\left(\prod_{r+1\leq u\neq q\leq k} v_q\right).
\end{equation*}
We consider the set $S:=\{\varepsilon_{\beta}(r_1),\dots,\varepsilon_{\beta}(r_k)\}$ in ${\rm Sp}_4(R).$ Then clearly $\Pi(S)=\emptyset$ holds. Further, according to the proof of \cite[Theorem~6.3]{General_strong_bound}, the abelianization map $A:{\rm Sp}_4(R)\to{\rm Sp}_4(R)/[{\rm Sp}_4(R),{\rm Sp}_4(R)]=\mathbb{F}_2^{r(R)}$ is uniquely defined through $A(\varepsilon_{\phi}(x))=(x+\C P_1,\dots,x+\C P_{r(R)})$ for all $x\in R$ and $\Phi\in C_2.$ But then $S$ clearly maps to a generating set of $\mathbb{F}_2^{r(R)}.$ Thus Theorem~\ref{classifying_normal_generating_subsets_better} implies that $S$ is indeed a normally generating set of ${\rm Sp}_4(R)$.
Next, consider the map 
\begin{equation*}
\pi:{\rm Sp}_4(R)\to\prod_{i=1}^r{\rm Sp}_4(R/\C P_i)\times\prod_{j=r+1}^k{\rm Sp}_4(R/V_j),X\mapsto(\pi_{\C P_1}(X),\dots,\pi_{\C P_r}(X),\pi_{V_{r+1}}(X),\dots,
\pi_{V_k}(X))
\end{equation*}
for the $\pi_{\C P_i}:{\rm Sp}_4(R)\to{\rm Sp}_4(R/\C P_i)$ and $\pi_{V_j}:{\rm Sp}_4(R)\to{\rm Sp}_4(R/V_j)$ being the reduction homomorphisms.
Then note that 
\begin{align*}
\|{\rm Sp}_4(R)\|_S&\geq\|\left(\prod_{i=1}^r {\rm Sp}_4(R/\C P_i)\right)\times\left(\prod_{j=r+1}^k {\rm Sp}_4(R/V_j)\right)\|_{\pi(S)}\\
&=\left(\sum_{i=1}^r\|{\rm Sp}_4(R/\C P_i)\|_{\varepsilon_{\beta}(r_i+\C P_i)}\right)
+\left(\sum_{j=r+1}^k \|{\rm Sp}_4(R/V_j)\|_{\varepsilon_{\beta}(r_j+ V_j)}\right).
\end{align*}
So to finish the proof, it suffices to apply Proposition~\ref{dimension_counting_sln_sp_2n} to obtain 
\begin{equation*}
\|{\rm Sp}_4(R/\C P_i)\|_{\varepsilon_{\phi}(x_i+\C P_i)}\geq 5\text{ and }\|{\rm Sp}_4(R/V_j)\|_{\varepsilon_{\beta}(r_j+ V_j)}\geq 4
\end{equation*}
for all $i=1,\dots,r$ and $j=r+1,\dots,k$. 
\end{proof}

\section*{Closing remarks}

We should also mention that contrary to how similar the question of strong boundedness of $G_2(R)$ appeared to ${\rm Sp}_4(R)$ in \cite{General_strong_bound}, it is actually easier to determine $\Delta_k(G_2(R))$ than $\Delta_k({\rm Sp}_4(R))$. This is mainly due to the fact that the 'intermediate' normal subgroup we construct in the proof of \cite[Theorem~5.13]{General_strong_bound} is not the $2R$-congruence subgroup $N_{G_2,2R}$ but instead the bigger group 
\begin{equation*}
N_{G_2}=\dl\{\varepsilon_{\phi}(2x)|x\in R,\phi\in G_2\text{ short}\}\cup\{\varepsilon_{\phi}(x)|x\in R,\phi\in G_2\text{ long}\}\dr.
\end{equation*}
We might address this in a future paper.

\bibliography{bibliography}

\def\polhk#1{\setbox0=\hbox{#1}{\ooalign{\hidewidth
  \lower1.5ex\hbox{`}\hidewidth\crcr\unhbox0}}}
  \def\polhk#1{\setbox0=\hbox{#1}{\ooalign{\hidewidth
  \lower1.5ex\hbox{`}\hidewidth\crcr\unhbox0}}}
  \def\polhk#1{\setbox0=\hbox{#1}{\ooalign{\hidewidth
  \lower1.5ex\hbox{`}\hidewidth\crcr\unhbox0}}}
  \def\polhk#1{\setbox0=\hbox{#1}{\ooalign{\hidewidth
  \lower1.5ex\hbox{`}\hidewidth\crcr\unhbox0}}}
  \def\polhk#1{\setbox0=\hbox{#1}{\ooalign{\hidewidth
  \lower1.5ex\hbox{`}\hidewidth\crcr\unhbox0}}}
  \def\polhk#1{\setbox0=\hbox{#1}{\ooalign{\hidewidth
  \lower1.5ex\hbox{`}\hidewidth\crcr\unhbox0}}} \def\cprime{$'$}
\begin{thebibliography}{10}

\bibitem{MR439947}
Eiichi Abe and Kazuo Suzuki.
\newblock On normal subgroups of {C}hevalley groups over commutative rings.
\newblock {\em Tohoku Math. J. (2)}, 28(2):185--198, 1976.

\bibitem{MR244257}
H.~Bass, J.~Milnor, and J.-P. Serre.
\newblock Solution of the congruence subgroup problem for {${\rm
  SL}_{n}\,(n\geq 3)$} and {${\rm Sp}_{2n}\,(n\geq 2)$}.
\newblock {\em Inst. Hautes \'{E}tudes Sci. Publ. Math.}, (33):59--137, 1967.

\bibitem{brenner1978covering}
J.~L. Brenner.
\newblock Covering theorems for {FINASIG}s. {VIII}. {A}lmost all conjugacy
  classes in {$A_{n}$} have exponent {$\leq{}4$}.
\newblock {\em J. Austral. Math. Soc. Ser. A}, 25(2):210--214, 1978.

\bibitem{MR2509711}
Dmitri Burago, Sergei Ivanov, and Leonid Polterovich.
\newblock Conjugation-invariant norms on groups of geometric origin.
\newblock In {\em Groups of diffeomorphisms}, volume~52 of {\em Adv. Stud. Pure
  Math.}, pages 221--250. Math. Soc. Japan, Tokyo, 2008.

\bibitem{MR704220}
David Carter and Gordon Keller.
\newblock Bounded elementary generation of {${\rm SL}_{n}({\mathcal O})$}.
\newblock {\em Amer. J. Math.}, 105(3):673--687, 1983.

\bibitem{MR1162432}
Douglas~L. Costa and Gordon~E. Keller.
\newblock Radix redux: normal subgroups of symplectic groups.
\newblock {\em J. Reine Angew. Math.}, 427:51--105, 1992.

\bibitem{MR2819193}
{\'S}wiatos{\l}aw~R. Gal and Jarek K\k{e}dra.
\newblock On bi-invariant word metrics.
\newblock {\em J. Topol. Anal.}, 3(2):161--175, 2011.

\bibitem{MR0396773}
James~E. Humphreys.
\newblock {\em Linear algebraic groups, corrected fifth printing}.
\newblock Springer-Verlag, New York-Heidelberg, 1975.
\newblock Graduate Texts in Mathematics, No. 21.

\bibitem{Karni_paper}
S.~Karni.
\newblock Covering numbers of groups of small order and sporadic groups.
\newblock In {\em Products of conjugacy classes in groups}, volume 1112 of {\em
  Lecture Notes in Math.}, pages 52--196. Springer, Berlin, 1985.

\bibitem{KLM}
Jarek K{\k e}dra, Assaf Libman, and Ben Martin.
\newblock On boundedness properties of groups.
\newblock {\em In preparation.}, https://arxiv.org/abs/1808.01815.

\bibitem{lawther1998diameter}
R.~Lawther and Martin~W. Liebeck.
\newblock On the diameter of a {C}ayley graph of a simple group of {L}ie type
  based on a conjugacy class.
\newblock {\em J. Combin. Theory Ser. A}, 83(1):118--137, 1998.

\bibitem{MR1865975}
Martin~W. Liebeck and Aner Shalev.
\newblock Diameters of finite simple groups: sharp bounds and applications.
\newblock {\em Ann. of Math. (2)}, 154(2):383--406, 2001.

\bibitem{MR3822326}
Daniel~A. Marcus.
\newblock {\em Number fields}.
\newblock Universitext. Springer, Cham, 2018.
\newblock Second edition of [ MR0457396], With a foreword by Barry Mazur.

\bibitem{MR3892969}
Aleksander~V. Morgan, Andrei~S. Rapinchuk, and Balasubramanian Sury.
\newblock Bounded generation of {$\rm SL_2$} over rings of {$S$}-integers with
  infinitely many units.
\newblock {\em Algebra Number Theory}, 12(8):1949--1974, 2018.

\bibitem{MR2357719}
Dave~Witte Morris.
\newblock Bounded generation of {${\rm SL}(n,A)$} (after {D}. {C}arter, {G}.
  {K}eller, and {E}. {P}aige).
\newblock {\em New York J. Math.}, 13:383--421, 2007.

\bibitem{MR1697859}
J\"{u}rgen Neukirch.
\newblock {\em Algebraic number theory}, volume 322 of {\em Grundlehren der
  Mathematischen Wissenschaften [Fundamental Principles of Mathematical
  Sciences]}.
\newblock Springer-Verlag, Berlin, 1999.
\newblock Translated from the 1992 German original and with a note by Norbert
  Schappacher, With a foreword by G. Harder.

\bibitem{Nica}
B.~Nica.
\newblock On bounded elementary generation for sl n over polynomial rings.
\newblock {\em Isr. J. Math.}, pages 403--410, 2018.

\bibitem{MR3616493}
Robert Steinberg.
\newblock {\em Lectures on {C}hevalley groups}, volume~66 of {\em University
  Lecture Series}.
\newblock American Mathematical Society, Providence, RI, 2016.

\bibitem{MR1044049}
O.~I. Tavgen.
\newblock Bounded generability of {C}hevalley groups over rings of
  {$S$}-integer algebraic numbers.
\newblock {\em Izv. Akad. Nauk SSSR Ser. Mat.}, 54(1):97--122, 221--222, 1990.

\bibitem{explicit_strong_bound_sp_2n}
Alexander Trost.
\newblock Explicit strong boundedness for higher rank symplectic groups.
\newblock {\em https://arxiv.org/pdf/2012.12328.pdf}, 2020.

\bibitem{General_strong_bound}
Alexander Trost.
\newblock Strong boundedness of simply connected {S}plit {C}hevalley groups
  defined over rings.
\newblock {\em https://arxiv.org/pdf/2004.05039.pdf}, 2020, submitted.

\bibitem{PhD_thesis}
Alexander Trost.
\newblock {Q}uantitative aspects of normal generation of split {C}hevalley
  groups.
\newblock {\em PhD thesis}, University of Aberdeen 2020, supervised by B.
  Martin.

\bibitem{MR2161255}
Peter V\'{a}mos.
\newblock 2-good rings.
\newblock {\em Q. J. Math.}, 56(3):417--430, 2005.

\end{thebibliography}
\bibliographystyle{plain}

\end{document}